\documentclass[a4paper,11pt]{article}
\usepackage[utf8]{inputenc}
\usepackage{comment}
\usepackage[titletoc]{appendix} % for pretty printing of appendix in toc
\usepackage{xcolor}
\usepackage[pdfencoding=auto, psdextra]{hyperref}
\hypersetup{
    colorlinks=true,
    }
\usepackage{enumitem} % labels in enumerate
\usepackage{diagbox} 
\usepackage{ulem}%sout
\usepackage{amssymb}
\usepackage{amsmath}
\usepackage{amsthm}
\usepackage{mathtools}
\usepackage{stmaryrd} % For double brackets
\usepackage{setspace} % horizontal spacing arrays
\usepackage{authblk}

\theoremstyle{plain}
\newtheorem{theorem}{Theorem}[section]
\newtheorem{prop}[theorem]{Proposition}
\newtheorem{lemma}[theorem]{Lemma}

\newtheorem{conj}[theorem]{Conjecture}

\newtheorem{hyp}[theorem]{Assumption}

\newtheorem{rk}[theorem]{Remark}

\newcommand{\R}{\mathbb{R}}
\def\P{\mathbb{P}}
\newcommand{\E}{\mathbb{E}}

\newcommand{\ub}{\mathbf{u}}
\newcommand{\vb}{\mathbf{v}}
\newcommand{\C}{\mathcal{C}} 
\newcommand{\ind}{\mathbf{1}}
\newcommand{\setind}[1]{\mathbf{1}_{\left\{#1\right\}}}
\newcommand{\bbrackets}[2]{\llbracket #1, #2 \rrbracket}
\newcommand{\angles}[2]{\langle #1, #2 \rangle}

\title{Metacommunity persistence on spatially heterogeneous landscapes}
\author[1]{Manon Costa} 
\author[2]{Madeleine Kubasch\footnote{Corresponding author: madeleine.kubasch@polytechnique.edu}} 
\author[3]{Nicolas Loeuille}

\affil[1]{\small Université de Toulouse, INSA Toulouse, CNRS, Institut de Mathématiques de Toulouse, Toulouse, France}
\affil[2]{École polytechnique, Sorbonne Université, UPC, UPEC, CNRS, IRD, INRA, Institute of Ecology and Environmental Sciences, IEES, Paris, France}
\affil[3]{Sorbonne Université, UPC, UPEC, CNRS, IRD, INRA, Institute of Ecology and Environmental Sciences, IEES, Paris, France}
 
\begin{document}

\maketitle

\begin{abstract}
% 163 mots
    We are interested in the long-time behaviour of the ecological dynamics of two competing species in a spatially heterogeneous environment consisting of two habitat types. 
    Our goal is to provide conditions for the persistence of the two populations.\\
    First, we consider a spatially continuous model, formalized as an infinite-dimensional system of integro-differential equations. 
    We show that if each species would persist if it were alone, then mutual invasibility of each other's monospecific equilibrium is a sufficient condition for long time survival of both species. 
    Second, we introduce a finite-dimensional system of ordinary differential equations which approximate the spatial dynamics by averaging over a finite number of habitat types.
    We derive an analogous sufficient condition for stable coexistence, and show that in this case, there exists a positive coexistence equilibrium. \\
    Finally, we complete our theoretical result using a simulation study.
    Our results indicate that mutual invasibility also is a necessary condition for stable coexistence in both models. 
    In addition, we show that the finite-dimensional model underestimates species' persistance, which indicates that spatial heterogeneity promotes survival. 
    \bigskip
    
    \noindent \textbf{Code availability:} \url{https://gitlab.com/m-kubasch/metacommunity-simulations} 
    \smallskip

    
\end{abstract}

\section{Introduction}

Many human activities such as agriculture \cite{greenFarmingFateWild2005, loeuilleChapterSixEcoEvolutionary2013} or the creation of natural reserves \cite{diamondIslandDilemmaLessons1975, mayIslandBiogeographyDesign1975}, alter the environment's spatial heterogeneity. Indeed, new habitat types such as farmed or urban areas are introduced, which may either be concentrated in relatively aggregated areas, or make arise mosaics of natural and artificial habitat patches. Understanding the impact of spatial heterogeneity on biodiversity thus is key to design sustainable land use strategies. 

A suitable framework for exploring this issue is provided by metacommunity models, which describe the ecological dynamics of a pool of interacting species whose habitat is composed of several localities or patches that are connected by colonization, which depends on the species' dispersal ability and on competition among species \cite{mouquetCommunityPatternsSourceSink2003, leiboldMetacommunityConceptFramework2004}. While patches may be characterized by their spatial position \cite{ovaskainenSpatiallyStructuredMetapopulation2001}, their number is typically assumed to be finite, thus leading to a discrete spatial structure. 
\medskip

In this paper, we introduce a spatially continuous metacommunity model of two species $u$ and $v$ competing for available resources in a spatially heterogeneous environment $\Omega \subset\R^2$. More precisely, we consider an integro-differential model for the dynamics  of the two populations $u,v : \R_+ \times \Omega \to [0,1] $ which can be written as 
\begin{equation}
\label{eq:ide-system_dim2}\left\{
\begin{aligned}
&\partial_t u(t,x)  = - \tau(x) u(t,x) + (1-u(t,x)-v(t,x))\int_{\Omega} u(t,y) c(x,y) dy\\
&\partial_t v(t,x)  = - \sigma(x) v(t,x) + (1-u(t,x)-v(t,x))\int_{\Omega} v(t,y) \gamma(x,y) dy\\
&0\le u_0(x)+v_0(x)  \le 1.  \\ 
\end{aligned}\right.
\end{equation}
The parameter $\tau : \Omega\mapsto\R_+$ (resp. $\sigma$) represents the extinction rate of population $u$ (resp. $v$) depending on the environment, while the kernel $c:\Omega\times \Omega\mapsto\R_+$ (resp. $\gamma$) describes the colonization rate.
We will show that this model corresponds to a graphon-type limit of a discrete metacommunity model, where the population evolves on a network of patches. The microscopic model considered is an invasion-exclusion model in which we model only the presence of a species on a given patch and not its density. In particular, we assume that when a species is present in an environment, that environment cannot be colonized by another species (preemptive competition). Such colonization-extinction dynamics are frequently used in ecological literature 
\cite{levinsDemographicGeneticConsequences1969, mouquetCommunityPatternsSourceSink2003, leiboldMetacommunityConceptFramework2004} and can also be interpreted in epidemiology as \textit{SIS}-type models \cite{lajmanovich_deterministic_1976, DDZ22}. 

We will also consider a specific case where the environment can be decomposed into  $\Omega=A\cup N$ where $A$ stands for an agricultural environment and $N$ a natural one. Further assuming that the functions $\tau,\sigma,c$ and $\gamma$ are piecewise constant, Equation \eqref{eq:ide-system_dim2} admits a discrete space analogous. It is obtained by considering the dynamics of $u$ and $v$ on the subspaces $A$ and $N$. This leads to studying the solution $(u_A,u_N,v_A,v_N)$ to a 4-dimensional system of differential equations:
\begin{equation}
\label{eq:ode-system}
   \left\{ \begin{aligned}
    u_A' &= -\tau_A u_A + (p_A - u_A - v_A) (c_{AA}u_A + c_{AN} u_N), \\
    u_N' &= -\tau_N u_N + (p_N - u_N - v_N) (c_{NA}u_A + c_{NN} u_N), \\
    v_A' &= -\sigma_A v_A + (p_A - u_A - v_A) (\gamma_{AA}v_A + \gamma_{AN} v_N), \\
    v_N' &= -\sigma_N v_N + (p_N - u_N - v_N) (\gamma_{NA}v_A + \gamma_{NN} v_N). \\
    \end{aligned}\right.
\end{equation}
In this case, note that $(u_A,u_N)$ (resp. $(v_A,v_N)$) corresponds to the integrated density $u$ (resp.$v$) over $A$ and $N$, and that $p_A=|A|$ and $p_N=|N|$ are the respective sizes of the agricultural and natural areas. A more precise derivation is given below in Section \ref{subsec:derivation}. In particular, this discrete space model is a harlequin model, as both local extinction and colonization rates are entirely determined by the habitat types of the involved patches
\cite{hornCompetitionFugitiveSpecies1972, leiboldSpeciesSortingPatch2015}. 
\medskip

This article aims at obtaining a persistence extinction criterion for model \eqref{eq:ide-system_dim2} and for its discrete space analog \eqref{eq:ode-system}. From the point of view of applications, this amounts to predicting which species survive in a given environment, thus providing crucial information for conservation goals. 

Notably, the persistence of a single species in a spatially heterogeneous environment is already well understood. This question was studied from a epidemiological point of view, since the monospecific model can also depict the spread of an epidemic in a structured population. Persistence criteria have thus been established in mathematical epidemiology, both for the discrete \cite{lajmanovich_deterministic_1976} and continuous space versions \cite{DDZ22}. These articles prove that a single species persists if it is capable of invading the landscape when starting from an infinitesimal population. This is possible if the exponential growth rate of the metapopulation close to zero is positive. Equivalently, starting from a typical occupied patch, the species needs to colonize on average strictly more than one other patch before going locally extinct. Importantly, in continuous space, the rigorous study of the metapopulation's long-time behaviour makes use of the monotonicity of the underlying semi-flow, both in time (\textit{i.e.} all trajectories are monotone) and initial condition (\textit{i.e.} starting from two initial conditions such that one is everywhere greater than the other, this ordering is conserved over time) \cite{DDZ22}. 

When moving from metapopulation to metacommunity dynamics, there are two main difficulties. First, there is a larger panel of possible outcomes to be considered, as extinction may correspond either to extinction of a single population, or to extinction of both populations. Second, the dynamics are no longer monotonous since both species compete for available habitat, whereas cooperation occurs through colonization among local populations of a given species. 
\medskip

The article is organized as follows. First, in Section \ref{sec:model}, we derive the spatially continuous model \eqref{eq:ide-system_dim2} from a stochastic graph model. We further state our main results on metacommunity persistence obtaining a partial classification of the limiting behaviour of the community. We complete our theoretical results by exploring numerically the cases for which the long-time behaviour of Equation \eqref{eq:ide-system_dim2} remains unresolved. Next, in Section \ref{sec:harlequin}, we introduce the discrete space model \eqref{eq:ode-system} that arises as an approximation of the continuous space model. We obtain persistence results for the metacommunity  in that setting. We also evaluate through simulations whether the discrete space model provides a satisfying approximation of the spatially continuous model, when extinction and colonization rates are not piecewise constant. Finally, Sections \ref{sec:proofs_ide} and \ref{sec:proofs_hq} are devoted to proofs. %

\section{Model and main results}
\label{sec:model}

\subsection{Model derivation}
\label{subsec:derivation}
In this section, we derive the integro-differential model \eqref{eq:ide-system_dim2} from a stochastic model describing metacommunity dynamics on a finite spatial random graph. The aim of this step is twofold. On the one hand, it allows us to properly connect the integro-differential model to patch occupancy models, which are well established in theoretical ecology \cite{levinsDemographicGeneticConsequences1969,ovaskainenSpatiallyStructuredMetapopulation2001, mouquetCommunityPatternsSourceSink2003, leiboldMetacommunityConceptFramework2004}. On the other hand, this argument also establishes existence and uniqueness of the solution to Equation \eqref{eq:ide-system_dim2}. 

This will be achieved for a more general class of metacommunity models, which are of interest \textit{per se}, and of which the integro-differential model \eqref{eq:ide-system_dim2} arises as a special case. 

Before proceeding, let us introduce some general notations. For any integers $n \leq m$, we write $\bbrackets{n}{m} = \{n, \dots, m\}$. For a measurable space $(E, \mathcal{E})$, let $\mathcal{M}_1(E)$ designate the set of probability measures on $E$, and for $\mu \in \mathcal{M}_1(E)$ and $f$ an appropriate real-valued test function (either non-negative or bounded), we let $\angles{\mu}{f} = \int_E f(x) \mu(dx)$. Also, for $x \in E$, $\delta_x$ is the Dirac measure at $x$. Further, $\mathcal{B}_b(E, \R)$ corresponds to the set of bounded measurable functions $f : E \to \R$. Finally, given a Polish metric space $X$, $\mathbb{D}(\R_+, X)$ corresponds to the space of right-continuous left-limited (càdlàg) functions $\R_+ \to X$, endowed with the Skorokhod topology. 
Let $f$ and $g$ be two functions from $E$ to $\mathbb R$, we will denote $f\le g$ if for all $x\in E$, $f(x)\le g(x)$. The same notation will hold for two vectors $u,v\in\R^n$, we say $u\le v$ if $\forall 1\le i\le n$, $u_i\le v_i$.

%\madeleine{Il faut mentionner quelque part que nos preuves sont très proches de celles de Delmas et al 2024, la seule différence étant dans l'étape d'unicité puisque Delmas et co se servent de l'existence et de l'unicité des solutions à l'IDE pour conclure à l'existence et l'unicité des solutions à l'équation valeur mesure, tandis que nous on fait l'inverse.}
%\manon{J'ai mis une phrase au début de la preuve.}

\paragraph{Stochastic metacommunity model on finite networks}

Let $\Omega \subset \R^2$ be a compact set, representing the agricultural landscape, and $\mathcal{S} = \bbrackets{1}{S}$ the set of $S$ species of interest. Up to an appropriate rescaling, we can always consider that $|\Omega|=1$. The metacommunity will spread on a network connecting $K$ patches, each of which are characterised by a spatial position $(x_1, \dots, x_K)\in\Omega^K$. 
In addition, each patch has an occupation status $w_k$, being either empty ($w_k = 0$), or occupied by a given species $(w_k \in \mathcal{S})$. 

While the patches spatial positions do not change over time, their occupation status does, according to the following dynamics. Consider a patch of spatial position $x$, inhabited by a given species $w \in \mathcal{S}$. Two events may occur:
\begin{enumerate}
    \item \textbf{Local extinction: } The local population goes extinct at rate $\tau_w(x)$, in which case the patch becomes of type $(x, 0)$. 
    \item \textbf{Colonisation: } The local population colonises another patch of type $(y, w')$, at rate $K^{-1} c_{w, w'}(x,y)$. The arrival patch becomes of type $(y, w)$.
\end{enumerate}
We will work under the following assumption.
\begin{hyp}
\label{hyp:jump-rates}
\begin{enumerate}
    \item For any $w \in \mathcal{S}$, the application $x \in \Omega \mapsto \tau_w(x) > 0$.
    \item For any $w, w' \in \mathcal{S}$, the application $(x,y) \in \Omega^2 \mapsto c_{w,w'}(x,y) \geq 0$ is continuous and connected, \textit{i.e.} for any Borel set $A \subset \Omega$ such that $|A| > 0$ and $|A^C| > 0$, 
    \[ \int_{A \times A^C} c_{w,w'}(x,y) dx dy > 0.  \]
\end{enumerate}  
%For any $w, w' \in \mathcal{S}$, the applications $x \in \Omega \mapsto \tau_w(x) > 0$ and $(x,y) \in \Omega^2 \mapsto c_{w,w'}(x,y) \geq 0$ are continuous.
\end{hyp}

We are interested in studying the stochastic process $\eta^K_t$, which tracks the empirical type distribution of our set of patches at time $t \geq 0$. Informally, $\eta^K$ may be defined by 
\begin{equation*}
    \eta^K_t = \frac{1}{K} \sum_{k = 1}^K \delta_{(x_k, w_k(t))}.
\end{equation*}
%where we recall that $\delta_{(x, w)}$ designates the Dirac measure at $(x, w)$ on $E = \Omega \times \bbrackets{0}{S}$. Throughout the following, $\mathcal{M}_1(E)$ designates the space of probability measures on $E$. 
In particular, for any $t \geq 0$, $\eta^K_t \in \mathcal{M}_1(E)$.

In order to rigorously define $\eta^K$, we will characterize it as the unique solution to a stochastic differential equation driven by Poisson Point Measures. Consider $Q_0$ and $Q_1$ two independent Poisson Point Measures satisfying the following. $Q_0$ is defined on $\R_+ \times E_0$ with $E_0 = \bbrackets{1}{K} \times \R_+$ and of intensity 
\[\mu_0(ds, dk, d\theta) = ds \otimes \mu_\#(dk) \otimes d\theta,  \]
and $Q_1$ is defined on $\R_+ \times E_1$ with $E_1 =\bbrackets{1}{K}^2 \times \R_+$ and of intensity 
\[\mu_1(ds, dk, d\ell, d\theta) = ds \otimes \mu_\#(dk) \otimes \mu_\#(d\ell) \otimes d\theta,  \]
where $ds, d\theta$ designate the Lebesgue measure and $\mu_\#$ the counting measure on $\bbrackets{1}{K}$. We are now ready to properly define $\eta^K$. 

\begin{prop}
    Given $(x_1, \dots, x_K) \in \Omega^K$ and $(w_1(0), \dots, w_K(0)) \in \bbrackets{0}{S}^K$, define 
    \[ \eta^K_0 = \frac{1}{K} \sum_{k = 1}^K \delta_{(x_k, w_k(0))}. \]
    Under Assumption \ref{hyp:jump-rates}, $\eta^K$ is the unique strong solution to the following measure-valued stochastic differential equation:
    \begin{equation*}
        \begin{aligned}
            \eta^K_t &= \eta^K_0 + \frac{1}{K} \int_0^t \int_{E_0} \setind{\theta \leq \tau_{w_k(s-)}(x_k)} \big(\delta{(x_k, 0)} - \delta{(x_k, w_k(s-))}  \big)Q_0(ds, dk, d\theta) \\ 
            & + \frac{1}{K} \int_0^t \int_{E_1} \setind{\theta \leq K^{-1} c_{w_k(s-), w_\ell(s-)}(x_k, x_\ell)} \big(\delta{(x_\ell, w_k(s-))} - \delta{(x_\ell, w_\ell(s-))}  \big)Q_1(ds, dk,d\ell, d\theta).
        \end{aligned}
    \end{equation*}
\end{prop}
The proof is classical, as jump rates are bounded, and we refer to \cite{fournier_microscopic_2004} for detail.

\paragraph{Graphon limit}

We are interested in a scaling limit of the previous model when the number of patches grows to infinity. Since the network connecting the patches is dense from Assumption \ref{hyp:jump-rates}, this corresponds to a graphon limit. 
The main result of this section lies in the following Theorem, which states that under appropriate assumptions, the stochastic metacommunity model converges to a deterministic limit when the number of patches $K$ goes to infinity. This requires the following condition.
\begin{hyp}
\label{hyp:eta0}
    \begin{enumerate}
        \item The sequence of initial condition $(\eta^K_0)_{K \geq 1}$ converges in probability to $\zeta_0 \in \mathcal{M}_1(E)$, such that $\zeta_0(dx \times \bbrackets{0}{S})$ is absolutely continuous with respect to the Lebesgue measure, with density $\mu > 0$.
        \item There exists a family of functions $(u^0_i \in \mathcal{B}_b(\Omega, [0,1]), i \in \bbrackets{0}{S})$ such that $\sum_{i=0}^S u^0_i = 1$ and $\zeta_0$ satisfies 
        \begin{equation}
        \label{eq:zeta0}
            \zeta_0(dx, dw) = \mu(x)dx \sum_{i=0}^S u^0_i(x) \delta_i(dw).
        \end{equation}
        %\madeleine{Je ne suis pas sure de l'espace de fonction à considérer pour $u_0$, on peut peut-être relâcher à borélienne bornée au lieu de $C^0$ pour la condition initiale. Cela fait d'ailleurs écho à l'une de tes remarques plus loin.}
    \end{enumerate}
\end{hyp}

We are now ready to state the convergence result.
\begin{theorem}
\label{thm:lln}
    Under Assumptions \ref{hyp:jump-rates} and \ref{hyp:eta0}, the sequence $(\eta^K)_{K \geq 1}$ converges in probability in $\mathbb{D}(\R_+, \mathcal{M}_1(E))$ to the continuous deterministic measure-valued process $\zeta$ defined as follows:
    \begin{equation}
    \label{eq:def-zeta}
        \zeta_t(dx, dw) = \mu(x) dx \sum_{i = 0}^S u_i(t,x) \delta_i(dw),
    \end{equation}
    where the family of functions $(u_i, i \in \bbrackets{0}{S})$ is characterized as the unique solution to 
    \begin{equation}
    \label{eq:full-ide-system}
        \begin{aligned}
            \sum_{i=0}^S u_i & = 1,\quad u_i : \R_+ \times \Omega \to [0,1] \quad  \forall i \in \{0,\dots, S\}, \quad \text{and } \\ 
            \partial_t u_i(t,x) & = - \tau_i(x) u_i(t,x) + \sum_{j = 0}^S u_j(t,x) \int_{\Omega} u_i(t,y) c_{i,j}(x,y) \mu(y) dy\\
             &\phantom{=} - u_i(t,x) \sum_{j=1}^S  \int_{\Omega} u_j(t,y) c_{j,i}(x,y) \mu(y) dy, \quad \forall i \in \{1, \dots, S\}, 
         \end{aligned}
    \end{equation}
    with initial condition provided by $u_i(0,\cdot) = u^0_i$ for any $i \in \bbrackets{1}{S}$. 
\end{theorem}
The proof (given in Appendix \ref{app:preuve_graphon}) follows a tightness-uniqueness argument as in \cite{delmas_individual-based_2024} who studied a closely related model for a monomorphic population. The main difference of our proof and the one in \cite{delmas_individual-based_2024} lies in the fact that we deduce existence and uniqueness of the solution of \eqref{eq:def-zeta} from the convergence of the stochastic process whereas \cite{delmas_individual-based_2024} identifies the limit using the properties of the integro-differential system. 
%\paragraph{Proof of Theorem \ref{thm:lln}}

\subsection{Main results for the continuous space model} 
\label{sec:results}

Theorem \ref{thm:lln} establishes that Equation \eqref{eq:ide-system_dim2}:
\begin{equation*}
         \left\{ \begin{aligned} 
            &\partial_t u(t,x)  = - \tau(x) u(t,x) + (1-u(t,x)-v(t,x)) \int_{\Omega} u(t,y) c(x,y) dy\\
            &\partial_t v(t,x)  = - \sigma(x) v(t,x) + (1-u(t,x)-v(t,x)) \int_{\Omega} v(t,y) \gamma(x,y) dy
            \\
           & u(0,x)=u_0(x)\ge0, \quad v(0,x)=v_0(x)\ge0, \quad u_0+v_0 \le 1.
         \end{aligned}\right.  
\end{equation*}
emerges in the graphon limit of a stochastic patch occupancy model by considering a set of two species ($S = 2$), with uniform spatial distribution $\mu(x) = 1$ and the following choice of colonization kernel: 
    \begin{equation*}  
    \quad c_{i,j}(x,y) = 
    \begin{cases}
        0 & \text{if} \; i = 0 \text{ or } j > 0, \\
        c(x,y) & \text{if } i = 1 \text{ and } j = 0, \\
        \gamma(x,y) & \text{if } i = 2 \text{ and } j =0.
    \end{cases}
    \end{equation*}
This correspond to the case where populations can only invade an empty position, and the colonization rate is different for the two populations.
Theorem \ref{thm:lln} also ensures the existence of a unique solution to \eqref{eq:ide-system_dim2} satisfying $0\le u+v\le 1$.\\
We introduce an additional assumption which requires that the colonization kernels $c$ and $\gamma$ do not vanish. This assumption will be useful to ensure the existence of a positive equilibrium in the monospecific case $S = 1$ (see Theorem \ref{theo:mono}).
\begin{hyp}
\label{ass1}
In the following, we assume that 
\begin{itemize}
    \item The functions $\tau$ and $\sigma$ are continuous and positive on $\Omega$.
    \item The functions $c$ and $\gamma$ are continuous and positive on $\Omega\times \Omega$.
\end{itemize} 
\end{hyp}

%\manon{Il faudrait énoncer ici un résultat qui assure que si $u_0+v_0\le  1$ , les solutions sont bien définies et restent dans l'espace $$\Delta^2=\{ (f,g) \in C(\Omega,[0,1])^2, 0\le f+g\le 1\}.$$}
%\madeleine{C'est lié à la preuve de la limite graphon, le résumer ici ou dans la section précédente ?}

\paragraph{Monospecific case}
Let us first focus on the case where a single species lives in the environment $\Omega$. This case was previously studied by Delmas Dronnier and Zitt \cite{DDZ22} and we recall their main result.
\\
Consider a compact subset $\Omega \subset \R^2$,  as well as functions $p \in L^\infty(\Omega, [0,1])$, $\tau \in \mathcal{C}(\Omega, (0,+\infty))$, and $c \in \mathcal{C}(\Omega^2, (0, +\infty))$. Let $w_0 \in \mathcal{C}(\Omega, [0,1])$ and consider the following integro-differential equation: 
\begin{equation}
\label{eq:monospecific-ide}
% \begin{aligned}
\begin{cases}
    \partial_t w(t,x) &= -\tau(x) w(t,x) + (1 - p(x) - w(t,x)) \int_\Omega c(x,y)w(t,y)dy, \\
    \quad w(0,x) &= w_0(x).
\end{cases}
% \end{aligned}
\end{equation}
%\manon{Vérifier ici si on a besoin de la continuité de la condition initiale ou si l'hypothèse $L^\infty$ est suffisante. La continuité permet d'obtenir que le temps d'atteinte de l'espace $\Delta_p$ est uniformément borné.\\}
Note that in \cite{DDZ22}, this equation is used to model the spread of an epidemic in a vaccinated population (see Section 5 of their article). In that setting, $w$ represents the density of infected individuals, $\tau$ is the healing rate, $p$ the proportion of vaccinated individuals and $c$ the contact rate in the population.\\
In order to state their persistence-extinction result, let us introduce the first generation operator as the linear operator such that for any $g \in L^{\infty}(\Omega, \R_+)$:
\begin{equation}
    \label{eq:def_T}
    \mathcal{T}_{c/\tau} (g)(x) = \int_\Omega \frac{c(x,y)g(y)}{\tau(y)} dy,  
\end{equation} 
and a scaling operator
\begin{equation}
    \label{eq:def_S}
    \mathcal{S}_p (g)(x) = (1 - p(x))g(x).
\end{equation} 

% Expliquer ici que 
For any linear operator $A$ acting on $g \in L^{\infty}(\Omega, \R_+)$, we let $r(A)$ designate its spectral radius, i.e. 
\[ r(A) = \sup \{|\lambda| : \lambda \in \text{spectrum}(A) \}. \]
The asymptotic behaviour of the unique solution to Equation \eqref{eq:monospecific-ide} is entirely characterized by the following result. 

\begin{theorem}[\cite{DDZ22}, Theorem 1.5]  
\label{theo:mono}
Consider any initial condition $w_0 \in \mathcal{C}(\Omega, [0,1])$. Under Assumptions \ref{ass1}, the following assertions hold.
\begin{enumerate}[label=(\roman*)]
\item There exists a unique solution $w$ to \eqref{eq:ode-system}, well defined on $[0,\infty[$. %Furthermore if there exists $A\subset \Omega$ such that for all $x\in A$, $w_0(x)>p(x)$ then after a finite time $T_{A,w_0}$, the solution satisfies $w(t,x)\in[0,p(x)]$ for all $x\in\Omega$ and $t\ge T_{A,w_0}$.
\item  If $r(\mathcal{T}_{c/\tau} \circ \mathcal{S}_p) \leq 1$, then for any initial condition $w_0$, the solution $w_p$ to Equation \eqref{eq:monospecific-ide} converges uniformly to 0 as $t$ grows to infinity.
\item If $r(\mathcal{T}_{c/\tau} \circ \mathcal{S}_p) > 1$, there exists a unique equilibrium $\bar{w}_p$ different from $0$, and it is positive. Furthermore for any $w_0$ such that its integral is positive, the solution $w_p$ to Equation \eqref{eq:monospecific-ide} converges uniformly to $\bar{w}_p$ as $t$ grows to infinity. In addition, $0 < \bar{w}_p < 1 - p$ and if $p$ is continuous on $\Omega$, then $\bar{w}_p \in \mathcal{C}(\Omega, [0,1])$.
\end{enumerate}
\end{theorem}

\begin{proof}
Note that Theorem 1.5 in \cite{DDZ22} handles the case where $p=0$ and Section 5 makes the connection when $p\neq 0$. Indeed, let us remark that if $w_p$ is a solution to $\eqref{eq:ide-system_dim2}$, then $z(t,x)=\frac{w_p(t,x)}{1-p(x)}$ is a solution to the same system for $p=0$. %This correspondence allows to obtain the theorem.
\\
%The only remaining item concerns the case where the initial condition $w_0$ may be larger than $1-p$. This case is handled in Lemma \ref{lemma:con_ini} whose proof can be found in Section \ref{sec:proofs}. \\
%\madeleine{Preuve des prop ajoutées :}
Let us now focus on the stated properties of $\bar{w}_p$ in $iii)$. First, recall that $\bar{w}_p$ satisfies, for all $x \in \Omega$:
\[0 = -\tau(x) \bar{w}_p(x) + (1 - \bar{w}_p(x)) \int_\Omega c(x,y) \bar{w}_p(y) dy.  \]
Thus $0 < \bar{w}_p(x) < 1 - p(x)$ for any $x \in \Omega$, as the right-hand size is negative for $\bar{w}_p (x) = 1 - p(x)$. Similarly, the right-hand side is positive if $\bar{w}_p(x) = 0$ since $c > 0$ and Theorem 1.5 (iii) in \cite{DDZ22} shows that $\bar{w}_p$ has nonzero integral.

Regarding the continuity of $\bar{w}_p$, notice that it follows from Equation \eqref{eq:monospecific-ide} that for any $x \in \Omega$,
\[\bar{w}_p(x) = \frac{(1 - p(x)) \int_{\Omega}c(x,y) \bar{w}_p(y) dy}{(1 - p(x)) \int_{\Omega}c(x,y) \bar{w}_p(y) dy - \tau(x)}, \]
As the applications $p$, $\tau$ and $c$ are continuous by assumption, this concludes. %\madeleine{cf Rk 2.16 DDZ}
%Consider $A=\{x\in\Omega, w_0(x)> p(x)\}$, then as long as $w(t,x)\ge p(x)$ we have 
%$$\partial_t w(t,x)\le -\tau(x) w(t,x)$$
%and thus $$w(t,x)\le w_0(x) \exp(-\tau(x) t).$$
%As a consequence, we obtain that $w(t,x)\le p(x)$ for $t\le \frac{1}{\tau(x)}\ln(w_0(x)/p(x))=T(x)$. Note that $T$ is continuous over $A$ and thus bounded. Combining this with Proposition 2.7 in\cite{DDZ22} stating that $\{f, 0\le f\le p\}$ is invariant for the dynamics, we obtain $(i)$.
\end{proof}
\paragraph{Metacommunity model}

In order to simplify the notations, we introduce the first generation operators for populations $u$ and $v$,
for any $g \in L^{\infty}(\Omega, \R_+)$:
\begin{equation}
    \label{eq:operatorT}
\mathcal{T}_{u} (g)(x) = \int_\Omega \frac{c(x,y)g(y)}{\tau(y)} dy,\qquad \mathcal{T}_{v} (g)(x) = \int_\Omega \frac{\gamma(x,y)g(y)}{\sigma(y)} dy, \end{equation}

Our first result concerns the extinction of both population $u$ and $v$.
\begin{theorem}[\textbf{Complete extinction}]
\label{theo:extin}
Assume $r(\mathcal T_u)\le1$ and $r(\mathcal T_v)\le1$, then for any initial conditions $u_0,v_0$ with $u_0+v_0\le 1$, we have 
$$\lim_{t\to\infty} \| u(t,\cdot) \|_\infty = \lim_{t\to\infty} \| v(t,\cdot)\|_\infty = 0.$$
%Furthermore, the convergence is exponentially fast \manon{préciser}.
\end{theorem}
%\manon{En regardant dans DDZ, les vitesses de convergences exponentielles sont liées à $s(\mathcal T - \tau)$ et ne sont pas valables dans le cas critique. Du coup, je ne sais pas si on précise vraiment.}

The second result states that if a single population is able to survive, then as expected, it will converge to the same equilibrium as in the monospecific setting.
\begin{theorem}[\textbf{Single species extinction}]
\label{thm:single-extinct}
Assume $r(\mathcal T_u)\leq1$ and $r(\mathcal T_v) >1$. For any initial conditions $u_0,v_0$ with $u_0+v_0\le 1$, it holds that 
\begin{equation*}
    \lim_{t \to \infty} \| u(t, \cdot)\|_\infty = 0
\end{equation*}
and 
\begin{equation*}
    \forall x \in \Omega, \quad \lim_{t \to \infty} v(t,x) = \bar{v}(x).
\end{equation*}
%\madeleine{+ vitesse de convergence exponentielle vers 0 pour $u$ (à compléter).}
\end{theorem}

% \madeleine{Est-ce qu'on peut avoir de la convergence uniforme de $v$ vers $\bar{v}$ ?}

The most interesting result concerns persistence.
From Theorem \ref{theo:mono}, we have that if $r(\mathcal T_u)>1$, there exists a unique positive equilibrium for the population $u$ when $v_0=0$ that we denote by $\bar{u}$. Similarly, we denote by $\bar{v}$ the positive equilibrium for population $v$ when $u_0=0$. Our next result states that if population $u$ can invade $0$ and $\bar v$ and population $v$ can invade $0$ and $\bar u$, then both populations are uniformly persistent (see \cite{smith_dynamical_2011} for more information on persistence theory). 

\begin{theorem}[\textbf{Uniform strong persistence}]
\label{theo:persist}
Assume $r(\mathcal T_u)>1$ and $r(\mathcal T_v)>1$. If further $r(\mathcal T_u\circ S_{\bar v})>1$ and $r(\mathcal T_v\circ S_{\bar u}) >1$, then there exists $\epsilon>0$ such that for any initial conditions $u_0,v_0$ with $u_0+v_0\le 1$, 
$$\liminf_t u>\epsilon>0, \quad \liminf_t v>\epsilon>0.$$
\end{theorem}

Finally, we also prove that the system admits a bistable case, in which the system admits two locally stable equilibria.
\begin{prop}
\label{prop:bistable}
Assume  $r(\mathcal T_u)>1$, $r(\mathcal T_v)>1$ but $r(\mathcal T_u\circ\mathcal S_{\bar v}) \le 1$ and  $r(\mathcal T_v\circ\mathcal S_{\bar u}) \le 1$ then the equilibria $(\bar u,0)$ and $(0,\bar v)$ are locally stable. More precisely, there exists $\delta>0$ and $\eta>0$ such that for any initial condition $(u_0,v_0)$ satisfying $|\bar u - u_0| \le \delta \bar u$ and $|v_0|\le \eta$, then the solution $(u,v)$ converges toward $(\bar u ,0)$.\\ The converse case for the convergence to $(0,\bar v)$ is similar with modified values of $\delta$ and $\eta$.
\end{prop}
This bistability would lead from an ecological perspective to \textit{priority effects}, so that the assembly of the metacommunity depends on which species $u$ or $v$ comes first. The importance of such priority effects are well grounded in ecology \cite{graingerApplyingModernCoexistence2019, koffelCompetitionFacilitationMutualism2021}, but they are usually proposed based on bistability in local patches (eg, due to high interspecific competition in Lotka-Volerra type models).
\medskip

The proof of these results can be found in Section \ref{sec:proofs_ide} and relies on comparisons between the solution $(u,v)$ of  \eqref{eq:ide-system_dim2} with solutions of the monospecific case. Note that we do not completely classify the situation, since our results do not cover the scenario where $r(\mathcal T_u)>1$, $r(\mathcal T_v\circ \mathcal S_{\bar u})>1$, $r(\mathcal T_v)>1$ but $r(\mathcal T_u\circ\mathcal S_{\bar v}) \le 1$, in which case we expect that the solution will converge to $(0, \bar v)$. 
Furthermore the question of the existence of a positive equilibrium $(u^*, v^*)$ to \eqref{eq:ide-system_dim2} remains open. Standard techniques using fixed point results (see \cite{smith_dynamical_2011,thieme_global_2011} for example) are difficult to adapt to this case. These questions will be explored numerically in the next Section.

% \madeleine{Ajouter une remarque sur le lien avec la théorie moderne de la coexistence ? ici, on a montré que l'invasibilité mutuelle est une condition suffisante à la survie, mais il n'est pas clair s'il s'agit d'une condition nécessaire.}

\subsection{Numerical exploration of the remaining cases}
\label{subsec:numIDE}

\begin{table}[h!]
    \centering
    \begin{tabular}{|c|c|c|}
        \hline
         & \textbf{Metacommunity capacity} & \textbf{Long-time behaviour}  \\
        \hline
       Extinction (Ext) & $r(\mathcal T_u)\le 1$ and  $r(\mathcal T_v)\le 1$  &  Both $u$ and $v$ go extinct.  \\ 
       \hline 
       Single species  &$r(\mathcal T_u)>1$ and  $r(\mathcal T_v)\le 1$  &  $u$ converges to $\bar{u}$, $v$ goes extinct.  \\
     extinction (SSE)     &$r(\mathcal T_u)\le 1$ and  $r(\mathcal T_v)> 1$  & 
     $u$ goes extinct, $v$ converges to $\bar{v}$.\\
     \hline 
     Mutual & $r(\mathcal T_u) > 1$ and $r(\mathcal T_v) > 1$ and & Bistability \\
     uninvasibility (MUI) & $r(\mathcal T_u\circ S_{\bar v}) \le 1$ and $r(\mathcal T_v\circ S_{\bar u}) \le 1$ &  \\
      \hline
       Non mutual &  $r(\mathcal T_u\circ S_{\bar v})>1$ and $r(\mathcal T_v\circ S_{\bar u}) \le 1$  & \textit{Undetermined.} \\
        invasion (NMI)&  $r(\mathcal T_u\circ S_{\bar v}) \le 1$ and $r(\mathcal T_v\circ S_{\bar u}) > 1$  &  \\
        \hline
       Coexistence (Coex) & $r(\mathcal T_u\circ S_{\bar v})>1$ and $r(\mathcal T_v\circ S_{\bar u}) >1$ & Both $u$ and $v$ persist.  \\
       \hline 
    \end{tabular}
    \caption{Classification of the long-time behaviour of the IDE model \eqref{eq:ide-system_dim2}, based on Theorems \ref{theo:extin}, \ref{thm:single-extinct} and \ref{theo:persist}.}
    \label{tab:Cases_classif_1}
\end{table}

The aim of this section is to explore in simulations the main open questions that remain. First, we focus on the case where both $u$ and $v$ persist, in order to establish whether the dynamics converge to a (unique) coexistence equilibrium. Second, we aim at investigating the long-time behaviour of Equation \eqref{eq:ide-system_dim2} in the undetermined case of non mutual invasion, as defined in Table \ref{tab:Cases_classif_1}. Finally, we aim at analysing the persistence outcome in the mutual uninvasibility scenario, starting from initial conditions which may not satisfy the assumptions of Proposition \ref{prop:bistable}.

% In order to achieve this, we explore numerically one possible parameter space for Equation \eqref{eq:ide-system_dim2}, which is inspired by applications in ecology and agriculture \madeleine{[REF notre article]}. 

\subsubsection{Numerical setting}
\label{sec:numerical-setting}

Numerically investigating the long-time behaviour of Equation \eqref{eq:ide-system_dim2} requires to explore the parameter space of Equation \eqref{eq:ide-system_dim2}, which we recall is given by 
\[ \tau, \sigma \in \mathcal{C}(\Omega, \R_+) \quad \text{ and } \quad c, \gamma \in \mathcal{C}(\Omega^2, \R_+). \]
In practice, we will focus on a subspace of this parameter space, which is inspired by applications to ecology and agriculture that are further developed in \cite{kubaschPrettyGoodYields2026}. 

More precisely, we assume that spatial heterogeneity results from the fact that the landscape consists of a mixture of different habitat types, for instance natural or agricultural habitat. Each point $x \in \Omega$ thus is characterized by its habitat quality $h(x) \in [0, 1]$. In our model $h(x)$ measures the farming intensity, therefore smaller values of $h(x)$ indicate mostly preserved, natural habitat whereas higher values of $x$ correspond to habitats which closely resemble farmed land. 

Throughout the following, we assume that colonization rates are independent from habitat quality, and solely rely on the euclidean distance between the departure and arrival points. Habitat quality instead conditions local extinction rates, depending on how well the species is adapted to either extreme habitat type. A more detailed presentation of the model parameters and performed simulations is given below.  

%\madeleine{\textbf{Question -} En rédigeant, je réalise que dans cette partie de l'exploration numérique, $p_A$ ne joue de fait aucun rôle (il n'influence ni les taux d'extinction ni les noyaux de dispersion), seul $H$ a un impact. En soi cela me donne envie de n'en parler que dans la seconde partie (avec le modèle harlequin).  mais le critère d'arrêt des simulations est basé (à rétrospectivement injuste titre) sur la distinction espaces agricoles ou naturels, ce qui obligerait donc quand-même à parler de $p_A, p_N$... Dernier point :  de façon générale, en considérant des taux d'extinction basés sur le fBM, augmenter $p_A$ fait mécaniquement baisser le taux d'extinction moyen au sein des espaces agricoles (je pense que ce n'est pas très grave car on ne va pas chercher véritablement d'interprétation écologique liée à $p_A$, mais c'est une grande différence avec ce qu'on fait dans l'autre article donc je le souligne pour nous).}

\paragraph{Habitat quality.} Environmental heterogeneity is modeled by a fractional Brownian sheet $B_H$ on $\Omega = [0,1]^2$, of Hurst exponent $H \in (0,1)$. For $H < 1/2$, increments are negatively correlated whereas they are positively correlated for $H > 1/2$ \cite{mandelbrotFractionalBrownianMotions1968, wuGeometricPropertiesFractional2007}. Thus the higher the Hurst exponent, the more the landscape is spatially aggregated as illustrated by Figure \ref{fig:landscapes}. Subsequently, habitat quality is obtained by normalizing $B_H$ to take values between 0 and 1: 
\[ h(x) = \frac{B_H(x) + \min_{x \in \Omega} B_H(x)}{\max_{x \in \Omega} B_H(x) - \min_{x \in \Omega} B_H(x)}. \]

Here, we consider $ H \in \{0.25, 0.5, 0.75\}$ with nine landscape replica each, thus amounting to 27 landscapes in total.

\begin{figure}
    \centering
    \includegraphics[width=\textwidth]{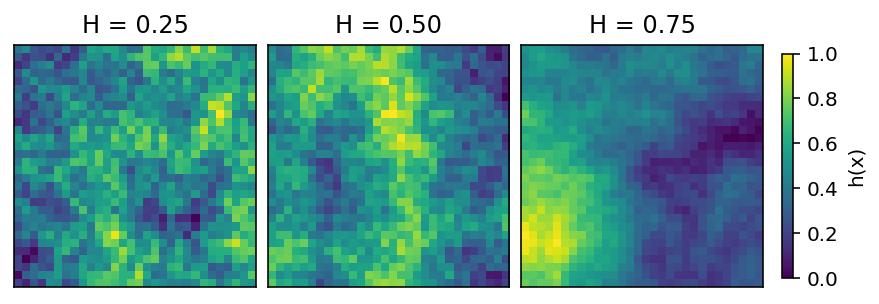}
    \caption{Three realization of $h(x)$, the habitat quality, for different values of Hurst exponent $H$, leading to increasing spatial aggregation (from left to right). Notice that $\Omega$ is discretized by a regular grid of shape $30 \times 30$.}
    \label{fig:landscapes}
\end{figure}

\paragraph{Colonization kernel.} Consider the family of exponential colonization kernels $\{c_\alpha : \Omega^2 \to (0,1]; \alpha > 0\}$, defined as follows. For $\alpha > 0$ and any $(x, y) \in \Omega^2$, 
\[ c_\alpha(x,y) = \exp(-\alpha \|x - y\|_2).  \]
In particular, the average dispersal distance equals $\alpha^{-1}$. In simulations, we let $\alpha \in \{20, 30\}$.% independently for each species. 

\paragraph{Local extinction rates.} Each species is characterized by its baseline extinction rates $(e_0, e_1) \in \{0.01, 0.006, 0.003\}^2$ in the extreme habitats of quality $0$ and $1$, respectively. 
%Each baseline extinction rate $e_i$ is defined as follows. Consider the metapopulation model \eqref{eq:monospecific-ide} in a homogeneous landscape, with extinction rate $e_{\star}$ and exponential dispersal kernel $c_{20}$: 
%\[\partial_t w(t,x) = -e_{\star} w(t,x) + (1 - w(t,x)) \int_\Omega c_{20}(x,y)w(t,y)dy. \]
 %Given $q \in (0,1)$, $e_i$ is defined as the unique value of $e_{\star}$ such that at equilibrium, $w$ occupies a fraction $q$ of the landscape. Here, we consider $q \in \{0.25, 0.5, 0.75\}$, leading to the corresponding extinction rates $e_i \in \{0.01, 0.006, 0.003\}$.
Given $(e_0, e_1)$, the species' local extinction rate $\tau$ or $\sigma$ at $x$ finally equals $e_0 + (e_1 - e_0) h(x)$.
%\manon{Je ne comprends pas ce paragraphe}
Note that with this choice, we have species with constant extinction rate when $e_0=e_1$, or species that are more adapted to a given habitat type. More precisely, when $e_1> e_0$, we will say that the species is more adapted to agricultural than natural habitat, and vice-versa.
The values for $\{e_0,e_1\}$ have been calibrated on the monospecific model with constant extinction rate (\textit{i.e.} in a homogeneous environment), in order to ensure that spatial occupation at equilibrium varies between 25 and 75\%.

\paragraph{Species pool.} From now on, each species is characterized by its parameters $(e_0, e_1, \alpha)$, leading to a total of 18 species. From this pool, we built 306 ordered couples of distinct species $(u,v)$ that will be used in the simulations.\\
Each species can further be classified according to two characteristics: 
\begin{itemize}
    \item Habitat preference: agricultural specialist ($e_1 < e_0$), natural specialist ($e_1 > e_0$) or generalist ($e_0 = e_1$). %\madeleine{Type 0, type 1 ou naturel, agricole ?}
    \item Dispersal capacity: $\alpha = 20$ or $\alpha = 30$.
\end{itemize}
We are now ready to describe the performed simulations.

\begin{figure}[ht]
    \centering
    \includegraphics[width = \textwidth]{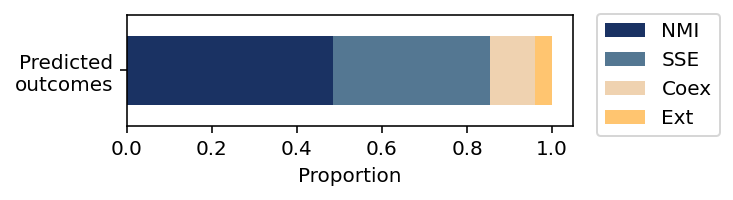}
    \caption{Observed frequencies of predictions in simulated data, based on the classification given in Table \ref{tab:Cases_classif_1}. More precisely given a landscape and a choice of parameters for the two populations, we computed the spectral radii of the associated operator to decipher the predicted outcome of the simulation. }
    \label{fig:fqcies}
\end{figure}

\paragraph{Simulation study.} Combining both the landscape realizations and the species couples $(u,v)$, the total simulation effort amounts to 8262 scenarios. For each scenario, we have  performed the following numerical experiments:
\begin{enumerate}[label=(\roman*)]
    \item Computation of $r(\mathcal{T}_u)$ and $r(\mathcal{T}_v)$, as well as $r(\mathcal{T}_u \circ \mathcal{S}_{\bar v})$ and $r(\mathcal{T}_v \circ \mathcal{S}_{\bar u})$ when appropriate. This allows to classify the simulations according to Table \ref{tab:Cases_classif_1}.
    \item Simulation of the IDE model to determine its long-time behaviour. We will in particular observe whether the dynamics converge toward a unique equilibrium.
\end{enumerate}

In particular, for (ii),  
simulations start from a homogeneous initial condition $u(0, \cdot) = v(0, \cdot) = 1/3$.
We simulate Equation \eqref{eq:ide-system_dim2} up to time $T_{\max} = 1000000$, unless there exists $k \in \{1,\cdots,50\}$  such that 
\begin{equation}
    \label{eq:cvgc-crit}
    \begin{aligned}
    \max_{s,t \in [2000(k-1), 2000k]} \|u(s,\cdot) - u(t,\cdot)\|_\infty & < 0.001 \\ 
    \text{and } \max_{s,t \in [2000(k-1), 2000k]} \|v(s,\cdot) - v(t,\cdot)\|_\infty & < 0.001.
    \end{aligned}
\end{equation}
In the latter case, we consider that the IDE has converged to an equilibrium. Notably, this has occurred in all performed simulations. 

In practice, this requires the discretization of the IDE, and of the integral operators $r(\mathcal{T}_u \circ \mathcal{S}_{p}), r(\mathcal{T}_v \circ \mathcal{S}_{p})$. Here, we approximate the square $[0,1]^2$ by a regular grid of shape $30 \times 30$, and discretize the IDE and integral operators accordingly. 
\smallskip
    
Figure \ref{fig:fqcies} depicts the number of occurrences of each case of Table \ref{tab:Cases_classif_1} observed in simulations. Notably, we did not observe any mutual uninvasibility scenarios, and we thus refer to forthcoming Section \ref{sec:mui} for further exploration of this regime. 
However, we have a satisfactory coverage of both coexistence and non-mutual invasion scenarios. We are thus ready to turn to the analysis of those simulations.

\subsubsection{Coexistence equilibrium}
\label{sec:num-coex}

\begin{figure}
    \centering
    \includegraphics[width = 0.5 \textwidth]{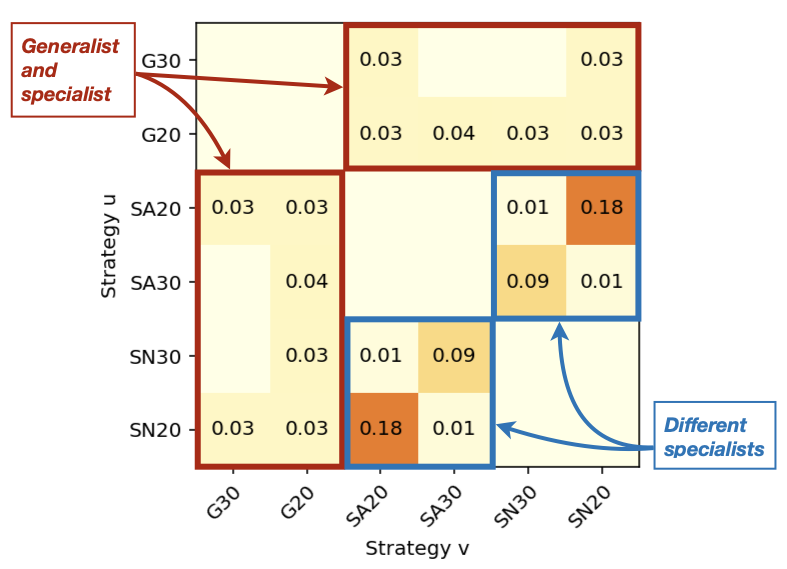}
    \caption{Proportion of coexisting strategies among all coexistence cases. Empty cases equal zero. SA20: agricultural specialist with dispersal capacity $\alpha = 20$, SN20: natural specialist with $\alpha = 20$, G20: generalist with $\alpha = 20$. SA30, SN30 and G30 are defined analogously with $\alpha = 30$.}
    \label{fig:coex-strategies}
\end{figure}

Let us first consider the coexistence scenarios, for which we want to determine whether the metacommunity dynamics converge to a coexistence equilibrium. Further, if this equilibrium exists, we would like to know whether it is unique. 
We have sampled 10 scenarios among those scenarios for which coexistence occurred for all landscape replicas (558 out of 882 coexistence outcomes). For each of those scenarios, we sample 10 random initial conditions as follows. At each point, the probability that the point $x \in \Omega$ is occupied follows a uniform distribution on $[0,1]$. In addition, given that $x$ is not empty, the probability that it is occupied by u also is uniformly distributed on $[0,1]$. Finally, for each initial condition, we simulate Equation \eqref{eq:ide-system_dim2} up to time $T_{\max} = 1000000$, unless an equilibrium is reached prior to that time according to criterion \eqref{eq:cvgc-crit}. Importantly, such convergence has occurred in all performed simulations. 
Further, for each scenario, we compute the maximum of the $L_\infty$-distance of the attained equilibria to the equilibrium reached by the first simulation of that scenario. In all cases, this distance is at most of order $10^{-4}$. This evidence indicates that in the coexistence case, the metacommunity converges to a unique coexistence equilibrium. 

Finally, we explore the species parameters allowing coexistence, and regroup them on Figure \ref{fig:coex-strategies}. We observe that two main patterns. The first case is the coexistence of two specialists of different habitat types, each of them occupying its preferred area. In the second case, a generalist and specialist share the habitat. This coexistence is facilitated when the generalist disperses better ($\alpha=20$) which is consistent with the classic prediction that dispersal helps species survival \cite{levinsDemographicGeneticConsequences1969}.

\subsubsection{Non mutual invasion}
\label{sec:nmi}

\begin{figure}
    \centering
    \includegraphics[width = 0.85\textwidth]{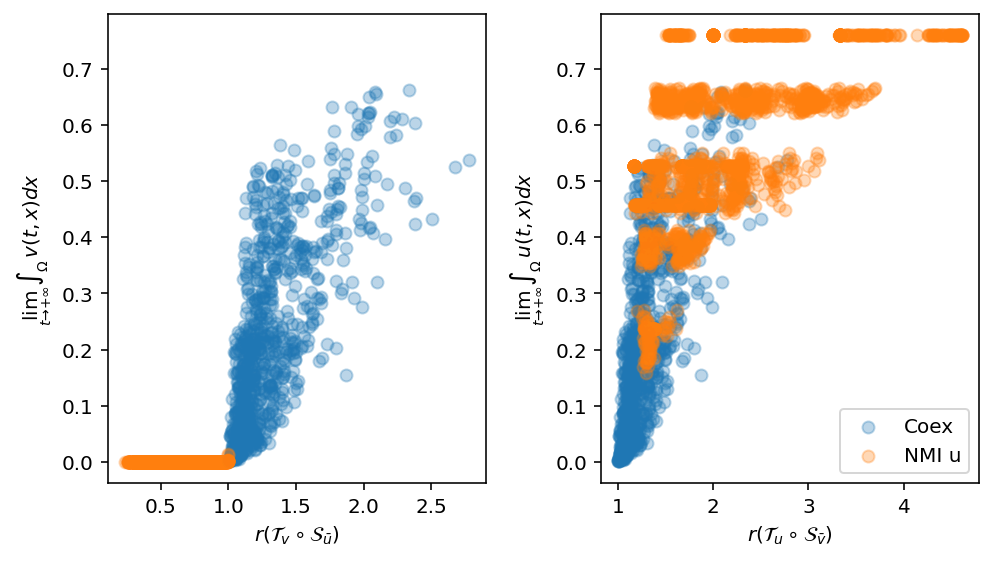}
    \caption{Numerical assessment of metacommunity persistence if both species invade zero, $u$ invades $\bar{v}$ and either $v$ invades $\bar{u}$ (Coex) or not (NMI u). (Left) Asymptotic abundance of $v$ as a function of $r(\mathcal T_v \circ \mathcal S_{\bar u})$. (Right) Asymptotic abundance of $u$ as a function of $r(\mathcal T_u \circ \mathcal S_{\bar v})$.}
    \label{fig:uq}
\end{figure}

Second, we focus on the case where both species are capable of invading the empty landscape, but only one species invades the monospecific equilibrium of the other. As the role of both species is symmetric, we will focus only on the case where $r(\mathcal{T}_u \circ \mathcal{S}_{\bar v}) > 1$, and $r(\mathcal{T}_v) > 1$ but $r(\mathcal{T}_v \circ \mathcal{S}_{\bar u}) \leq 1$. 

First, we have considered all scenarios for which $r(\mathcal{T}_u \circ \mathcal{S}_{\bar v}) > 1$ and $r(\mathcal{T}_v) > 1$. As shown in Figure \ref{fig:uq}, species $u$ always persists in this case. As established in Theorem \ref{theo:persist}, the condition $r(\mathcal{T}_v \circ \mathcal{S}_{\bar u}) > 1$ is sufficient to ensure that $v$ persists. Figure \ref{fig:uq} further suggests that this condition is necessary, as the total abundance of $v$ converges to zero as soon as $r(\mathcal{T}_v \circ \mathcal{S}_{\bar u}) \leq 1$.

Second, we check whether the metacommunity converges to an equilibrium in the non-mutual invasion scenario. As species $v$ does not persist, we expect the metacommunity to converge to $(0, \bar{u})$. Notice that $\bar{u}$ can be approximated numerically through simulations of the metapopulation model, allowing to test this assumption numerically. As in Section \ref{sec:num-coex}, we sample 10 scenarios out of all non-mutual invasion scenarios for which $u$ invades $\bar{v}$ for all landscape replica (1782 out of 2003 cases). Next, we sample 10 initial conditions at random and numerically check whether the metacommunity converges to an equilibrium. Again, all simulations have converged to an equilibrium. In addition, the distance from those equilibria to $(\bar{u},0)$ is at most of order $10^{-3}$.
\smallskip

Taken together, these observations lead to the following conjecture.

\begin{conj}
\label{conj:nmi}
    Assume $r(\mathcal T_u\circ S_{\bar v})>1$, $r(\mathcal{T}_v) > 1$ and $r(\mathcal T_v\circ S_{\bar u}) < 1$. For any initial conditions $u_0,v_0$ with $u_0+v_0\le 1$, it holds that 
    \begin{equation*}
    \lim_{t \to \infty} \| v(t, \cdot)\|_\infty = 0
    \end{equation*}
    and 
    \begin{equation*}
        \forall x \in \Omega, \quad \lim_{t \to \infty} u(t,x) = \bar{u}(x).
    \end{equation*}
\end{conj}
Naturally, the analogous result holds is the roles of $u$ and $v$ are interchanged. 
\smallskip

Together with Theorem \ref{theo:persist}, this conjecture implies that mutual invasibility is a necessary and sufficient condition for stable coexistence of both species. In particular, this is consistent with the predictions of the modern theory of coexistence \cite{chessonMechanismsMaintenanceSpecies2000}. 

Rigorously establishing Conjecture \ref{conj:nmi} however appears to be difficult. More precisely, the issue lies in the extinction of $v$. Given $r(\mathcal{T}_v) > 1$, we know that $v$ is capable of invading the empty landscape, \textit{i.e.} $v$ would persist on its own and thus must be driven to extinction through competition with $u$. Since $r(\mathcal{T}_v \circ \mathcal{S}_{\bar u}) < 1$, we know that if $u$ is absorbed by a close enough neighbourhood of $\bar{u}$, $v$ will not be able to persist. The main difficulty thus consists in ensuring that $u$ gets close to $\bar{u}$, and subsequently remains there. Proceeding as in the Proof of Theorem \ref{theo:persist}, it is possible to show that $u$ persists since $r(\mathcal T_u\circ S_{\bar v})>1$. However, further characterizing its long-time behaviour eludes this approach. 

\subsubsection{Mutual uninvasibility}
\label{sec:mui}

As we do not observe any mutual uninvasibility (MUI) in the numerical exploration detailed above, we conduct further simulations. More precisely, we consider a similar setting as introduced in Section \ref{sec:numerical-setting}, with randomly sampled model parameters:
\begin{itemize}
    \item Spatial aggregation $H$ is assumed to be uniformly distributed between 0 and 1. 
    \item For each species, the average dispersal distance $\alpha$ is sampled independently from a uniform distribution on $(1/30, 1/20)$;
    \item For each species, the baseline extinction rates $(e_0, e_1)$ are sampled independently from a uniform distribution on $(0.003, 0.01)^2$.
\end{itemize}

We consider a set of 20000 independently sampled parameter sets. For each of them, we draw a single realization of habitat quality of appropriate spatial aggregation $H$, and check whether the IDE model predicts mutual uninvasibility. Out of all 20000 parameter sets, we observe no such scenario. Mutual uninvasibility thus appears to be rare in this spatially explicit setting, contrary for instance to the classical competitive Lotka-Volterra model for which a quarter of the parameter choices leads to bistability \cite{Murray2002}. Naturally, their might also be an effect of the choice of parametrization of colonization kernels and extinction rates. In Section \ref{subsec:IDEvsharlequin_num}, we thus explore the same question in a simplified setting.

%Out of all 20000 parameter sets, we observe only one such scenario. Mutual uninvasibility thus appears to be rare in this spatially explicit setting, contrary for instance to the classical competitive Lotka-Volterra model for which a quarter of the parameter choices leads to bistability \cite{Murray2002}. 
%Note that the observed MUI scenario corresponds to a near-critical regime, with 
%\[ 0.98 < r(\mathcal T_u), \, r(\mathcal T_v),\, r(\mathcal T_u\circ S_{\bar v}),\, r(\mathcal T_v\circ S_{\bar u}) < 1.01 \] 
%and small monospecific equilibria ($\|\bar u \|_\infty$, $\| \bar v \|_\infty$ of order  $10^-3$). 

%We sample 30 initial conditions at random following the procedure detailed in Section \ref{sec:num-coex}. In each case, the metacommunity converges to one of the monospecific equilibria. This completes the result of Proposition \ref{prop:bistable}, which establishes bistability starting from initial conditions close to either monospecific equilibrium.  

\section{Comparison of the IDE model with a simplified harlequin model}
\label{sec:harlequin}

In this section, we aim at studying the difference between the spatially explicit IDE model \eqref{eq:ide-system_dim2} and a simplified spatial setting where the landscape is decomposed into two habitat types which entirely determine the event rates: there exist $\Omega_A$, $\Omega_N$ such that $$\Omega = \Omega_A \sqcup \Omega_N.$$
In order to keep track of the dynamics we consider the  averaged values of the densities $u,v$ over the two types of spaces. For $X \in \{A,N\}$, we let 
\begin{equation}
\label{def:u_x}
    \begin{aligned}
        u_X(t) = \int_{\Omega_X} u(t,x) dx, \text{ and } v_X(t) = \int_{\Omega_X} v(t,x) dx.
    \end{aligned}
\end{equation}
\subsection{The harlequin model}
Let us first consider an ideal setting, called the \textit{harlequin model}, in which the rate functions are constant by part. More precisely for all $x, y \in \Omega$:  
    \begin{equation}
    \label{eq:tau-arlequin}
         \tau(x) = \tau_A \ind_{\Omega_A}(x) + \tau_N \ind_{\Omega_N}(x) \text{ and }
         \sigma(x) = \sigma_A \ind_{\Omega_A}(x) + \sigma_N \ind_{\Omega_N}(x),
    \end{equation}
    and 
    \begin{equation}
    \label{eq:c-arlequin}
         c(x,y) = \sum_{X,Y \in \{A,N\}} c_{YX} \ind_{\Omega_X}(x) \ind_{\Omega_Y}(y) \text{ and }
         \gamma (x,y) = \sum_{X,Y \in \{A,N\}} \gamma_{YX} \ind_{\Omega_X}(x) \ind_{\Omega_Y}(y). \\
    \end{equation}

Integrating Equations \eqref{eq:ide-system_dim2} on $\Omega_A$ and $\Omega_N$ with such functions, yields the following dynamical system satisfied by $(\ub,\vb)=(u_A,u_N,v_A,v_N)$
\begin{equation*}
%\label{eq:ode-system}
    \left\{\begin{aligned}
    u_A' &= -\tau_A u_A + (p_A - u_A - v_A) (c_{AA}u_A + c_{AN} u_N), \\
    u_N' &= -\tau_N u_N + (p_N - u_N - v_N) (c_{NA}u_A + c_{NN} u_N), \\
    v_A' &= -\sigma_A v_A + (p_A - u_A - v_A) (\gamma_{AA}v_A + \gamma_{AN} v_N), \\
    v_N' &= -\sigma_N v_N + (p_N - u_N - v_N) (\gamma_{NA}v_A + \gamma_{NN} v_N). \\
    \end{aligned}\right.
\end{equation*}
where $p_X = |\Omega_X|$ for $X\in\{A,N\}$.\\

%\paragraph{Theoretical results for the harlequin model}
Similarly to the spatially explicit setting, we obtain a persistence result for the harlequin model. In order to achieve this, we first consider the case of a single species before turning to the metacommunity persistence criterion. 

\paragraph{Monospecific discrete space model}
Here we summarize the main results of Lajmanovich and Yorke \cite{lajmanovich_deterministic_1976}, who have studied the one-species equivalent of dynamical system \eqref{eq:ode-system}, namely: 
\begin{equation*}
    \left\{\begin{aligned}
    w_A' &= -\tau_A w_A + (p_A - w_A) (c_{AA}w_A + c_{AN} w_N), \\
    w_N' &= -\tau_N w_N + (p_N - w_N) (c_{NA}w_A + c_{NN} w_N). \\
    \end{aligned}\right.
\end{equation*}
In particular, the system is positively invariant in the set 
\[\Delta = \{w_A, w_N : 0 \leq w_X \leq p_X \; \forall X \in \{A,N\}\}.\] 
The authors show that the long-term dynamics of the system are entirely determined by the following quantity. Let 
\begin{equation*}
    M = \begin{pmatrix}
    p_A c_{AA} - \tau_A & p_A c_{AN} \\
    p_N c_{NA} & p_N c_{NN} - \tau_N
    \end{pmatrix},
\end{equation*}
and define $s$ as the spectral bound operator, \textit{i.e.}
\[ s(M) = \sup \{ \text{Re}(\lambda) : \lambda \in \text{spectrum}(M) \}. \]
largest real part of the eigenvalues of $M$. 
\begin{theorem}[Theorem 3.1 in \cite{lajmanovich_deterministic_1976}]
\label{thm:Laj}
Either $s(M) \leq 0$, in which case $0$ is globally asymptotically stable in $\Delta$, or $s(M) > 0$ and the system converges to an equilibrium $\overline{\ub} \in \Delta \setminus \{0\}$ which is globally asymptotically stable in $\Delta\setminus\{0\}$.
\end{theorem}

\begin{rk}
The criterion from Lajmanovich and Yorke \cite{lajmanovich_deterministic_1976} relies on the exponential growth of the population being either $>0$ or $\le 0$, whereas the criterion from Delmas, Dronnier and Zitt \cite{DDZ22} is expressed as a function of the metapopulation capacity being $>1$ or $\le 1$. Actually both criteria are two side of the same coin as specified in Section 4 in \cite{DDZ22}.\\
To be more specific, let us remark that for $\tau$ and $c$ as in \eqref{eq:tau-arlequin} and \eqref{eq:c-arlequin}, and a function
$g(x) = \alpha \setind{x \in \Omega_A} + \beta \setind{x \in \Omega_N} $, we obtain a matrix representation of the operator $\mathcal T_{c/\tau}$: for $g = (\alpha, \beta)$, 
\[ \mathcal T_{c/\tau} g = \begin{pmatrix}
p_A \frac{c_{AA}}{\tau_A} & p_N \frac{c_{NA}}{\tau_N} \\
p_A \frac{c_{AN}}{\tau_A} & p_N \frac{c_{NN}}{\tau_N} \\
\end{pmatrix} g.\]
Then Proposition 4.1 in \cite{DDZ22} proves that $r(\mathcal T_{c/\tau})-1$ and $s(\mathcal T_c - \tau Id)=s(M)$ have the same sign, which provides the equivalence between the two criterion. 
\end{rk}

\paragraph{Metacommunity persistence in discrete space}

We are now ready to turn towards the metacommunity persistence criterion. For $\vb = (v_A, v_N)$ such that $0 \leq v_X \leq p_X$ for $X \in \{A,N\}$, define
\begin{equation*}
    M_u(\vb) = \begin{pmatrix}
    (p_A - v_A) c_{AA} - \tau_A & (p_A - v_A) c_{AN} \\
    (p_N - v_N) c_{NA} & (p_N - v_N) c_{NN} - \tau_N
    \end{pmatrix}.
\end{equation*}
Analogously, for $\ub = (u_A, u_N)$ such that $0 \leq u_X \leq p_X$ for $X \in \{A,N\}$, we let 
\begin{equation*}
    M_v(\ub) = \begin{pmatrix}
    (p_A - u_A) \gamma_{AA} - \sigma_A & (p_A - u_A) \gamma_{AN} \\
    (p_N - u_N) \gamma_{NA} & (p_N - u_N) \gamma_{NN} - \sigma_N
    \end{pmatrix}.
\end{equation*}
Throughout the following, we let $\overline{\ub}$ and $\overline{\vb}$ designate the non-trivial mono-specific equilibrium of each species, if it exists. 
We obtain the following classification of the long-time behaviour of the dynamical system. Notice that these results are not a direct consequence of the results that we have established for the spatially continuous model, as colonization kernels and extinction rates of the form \eqref{eq:c-arlequin} do not satisfy Assumption \ref{ass1}.

\begin{theorem}
\label{thm:persistence_ODE}
The persistence of the metacommunity system \eqref{eq:ode-system} is characterized as follows. 
\begin{enumerate}[label=(\roman*)]
    \item If $s(M_u(0)) \leq 0$ and $s(M_v(0)) \leq 0$, then $(0,0)$ is globally asymptotically stable.
    \item If $s(M_u(0)) \leq 0$ and $s(M_v(0)) > 0$, then $(0, \overline{\vb})$ is globally asymptotically stable. Analogously, if $s(M_u(0)) > 0$ and $s(M_v(0)) \leq 0$, then $(\overline{\ub}, 0)$ is globally asymptotically stable.
    \item If $s(M_u(0)) > 0$ and $s(M_v(0)) > 0$ and furthermore $s(M_u(\overline \vb)) > 0$ and $s(M_v(\overline \ub)) > 0$, the metacommunity is strongly uniformly persistent: there exists $\varepsilon>0$ such that for any initial condition $\ub_0 + \vb_0 \le 1$,
    \begin{equation*}
        \forall X \in \{A,N\}, \quad \limsup_{t \to \infty} u_X(t) \ge  \varepsilon \quad \text{and} \quad \limsup_{t \to \infty} v_X(t) \ge \varepsilon.
    \end{equation*}
    \item If $s(M_u(0)) > 0$ and $s(M_v(0)) > 0$ and furthermore $s(M_u(\overline \vb)) \le  0$ and $s(M_v(\overline \ub)) \le 0$ then the two equilibria $(\overline \ub,0)$ and $(0,\overline \vb)$ are locally stable.
\end{enumerate}
\end{theorem}
Note that similarly to the IDE system, our results do not cover the cases where both $s(M_u(0)) > 0$ and $s(M_v(0)) > 0$ but only one of the two species can invade the other's equilibrium, that is $s(M_u(\overline \vb))>0$ but $s(M_v(\overline \ub)) \le 0$ or the converse.

In addition to the persistence result, we also prove the existence of a non trivial positive equilibrium $(\ub^*, \vb^*)$.

\begin{theorem}
\label{thm:exist_eq_ODE}
Assume $s(M_u(0)) > 0$ and $s(M_v(0)) > 0$ and furthermore $s(M_u(\overline \vb)) > 0$ and $s(M_v(\overline \ub)) > 0$, then there exists a positive equilibrium $(\ub^*, \vb^*)$ which furthermore verifies that $ \ub^* \le \overline \ub $ and $\vb^* \le \overline \vb$.
\end{theorem}

%\madeleine{Je crois que $s(M_u(\bar{v}) > 0 \implies s(M_u(0)) > 0$ (pour alléger les énoncés)}

The proof relies on the construction of a compact set positively invariant for the dynamics, which contains none of the equilibria $(0,0)$, $(0,\overline \vb)$ or $(\overline \ub, 0)$.

\subsection{Numerical comparison of the IDE and harlequin models}
\label{subsec:IDEvsharlequin_num}

\begin{table}[ht]
    \centering
    \begin{tabular}{|c|c|c|}
        \hline
         & IDE model & harlequin model  \\
        \hline
       Extinction (Ext) & $r(\mathcal T_u)\le 1$ and  $r(\mathcal T_v)\le 1$  &  $s(M_u(0))\le 0$ and $s(M_v(0)) \le 0$  \\ 
       \hline 
       Single species  &$r(\mathcal T_u)>1$ and  $r(\mathcal T_v)\le 1$  &  $s(M_u(0) )> 0$ and $s(M_v(0)) \le 0$  \\
     extinction (SSE)     &$r(\mathcal T_u)\le 1$ and  $r(\mathcal T_v)> 1$  &  $s(M_u(0))\le 0$ and $s(M_v(0)) > 0$  \\
     \hline 
     Mutual & $r(\mathcal T_u) > 1$ and $r(\mathcal T_v) > 1$ and & $s(M_u(0))> 0$ and $s(M_v(0)) > 0$ and \\
     uninvasibility (MUI) & $r(\mathcal T_u\circ S_{\bar v}) \le 1$ and $r(\mathcal T_v\circ S_{\bar u}) \le 1$ & $s(M_u(\overline \vb))\le 0$ and $s(M_v(\overline \ub)) \le 0$ \\
      \hline
       Non mutual &  $r(\mathcal T_u\circ S_{\bar v})>1$ and $r(\mathcal T_v\circ S_{\bar u}) \le 1$  & $s(M_u(\overline \vb))>0$ and $s(M_v(\overline \ub))\le 0$\\
        invasion (NMI)&  $r(\mathcal T_u\circ S_{\bar v}) \le 1$ and $r(\mathcal T_v\circ S_{\bar u}) > 1$  & $s(M_u(\overline \vb))\le 0$ and $s(M_v(\overline \ub)>0$\\
        \hline
       Coexistence (Coex) & $r(\mathcal T_u\circ S_{\bar v})>1$ and $r(\mathcal T_v\circ S_{\bar u}) >1$ & $s(M_u(\overline \vb))>0$ and $s(M_v(\overline \ub))>0$  \\
       \hline 
    \end{tabular}
    \caption{Classification of predicted persistence outcomes for the IDE and harlequin models.}
    \label{tab:Cases_classif}
\end{table}

When the rate functions are not constant over $\Omega_A$ and $\Omega_N$, we want to approximate the averaged IDE model $(u_A,u_N,v_A,v_N)$ defined in \eqref{def:u_x} with the solution of a simplified harlequin system.
To do so, we need to partition artificially each landscape of Section \ref{subsec:numIDE} into an agricultural area $\Omega_A$ and a natural area $\Omega_N$ such that $\Omega = \Omega_A \sqcup \Omega_N$.  We naturally define averaged rates for $X\in\{A,N\}$ by
\[ \widehat \tau_X = \frac{1}{p_X}\int_{\Omega_X} \tau(x)dx \quad \text{and}\quad \widehat \sigma_X = \frac{1}{p_X}\int_{\Omega_X} \sigma(x)dx, \]
and for $X,Y\in\{A,N\}$ 
\[ \widehat c_{XY} = \frac{1}{p_X p_Y}\int_{\Omega_X}\int_{\Omega_Y} c(x,y)dxdy \quad \text{and}\quad  \widehat \gamma_{XY} = \frac{1}{p_X p_Y}\int_{\Omega_X}\int_{\Omega_Y} \gamma(x,y)dxdy.\]
With these rates, we construct a solution $(\widehat \ub,\widehat \vb)=(\hat u_A, \hat u_N, \hat v_A, \hat v_N) $ of \eqref{eq:ode-system} and we will study if this approximation gives a good prediction for the long time behaviour of the populations. 
The main interest would be to provide  a simplified model whose prediction only require to compute the eigenvalues of a two dimensional matrix.\smallskip

Here, we consider different partition possibilities based on level sets of habitat quality. 
More precisely, given $h$, we designate by $q_p$ its empirical quantile of order $p \in [0,1]$. With this notation, we fix $p_A \in (0,1)$ and define the partition at level $p_A$ by
\[ A = \{ x \in [0,1]^2 : h(x) \geq q_{1 - p_A} \} \quad \text{and} \quad N = \Omega \setminus A.\] 
In particular, $p_A = |A|$ thus corresponds to the proportion of farmed land. 

%Throughout the following, we consider $p_A \in \{0.25, 0.5, 0.75\}$. For each value of spatial aggregation $H$, three landscape replicas are partitioned at each level $p_A$. 

\paragraph{Impact of the partition}
In order to explore the impact of the value of $p_A$ on the long time behaviour of the population dynamics of the harlequin approximation, we sampled 20 scenarios uniformly at random and for each we studied the harlequin approximation associated with values for  $p_A\in\{0.1, 0.2, 0.3, 0.4, 0.5,0.6, 0.7, 0.8, 0.9\}$. 
In each case, we use our theoretical criterion to predict the persistence outcome (see Table \ref{tab:Cases_classif}). We observe that in 18 out of 20 cases, the prediction does not depend on the partition induced by $p_A$, and in the 3 remaining cases, the differences are due to the computation of small eigenvalues (of order at most $10^{-4}$).

\paragraph{Comparison of the prediction}
We now focus on the theoretical predictions for both the IDE model and its harlequin approximation in the case $p_A=0.5$. Recall that with Theorems \ref{theo:extin}, \ref{thm:single-extinct}, \ref{theo:persist} and \ref{thm:persistence_ODE}, as well as Conjecture \ref{conj:nmi}, we obtain the different cases enumerated in Table \ref{tab:Cases_classif}.

\begin{figure}[ht]
    \centering
    \includegraphics[width = \textwidth]{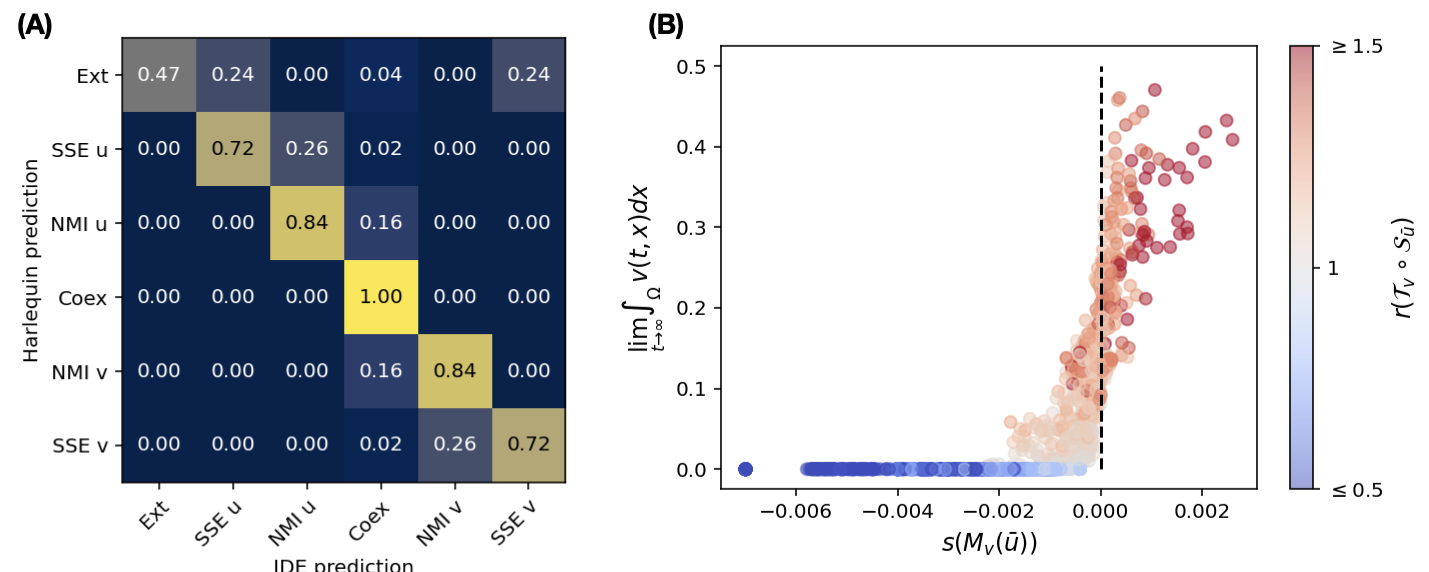}
    \caption{(A) Comparison of persistence outcomes predicted by the harlequin model (rows) and the IDE model (columns). Rows are normalized to sum to 1, so that each row indicates the probability that the IDE model predicts some outcome, given the prediction of the harlequin model. For instance, out of all scenarios leading to extinction for the harlequin model, only 45\% yield the same outcome when using the IDE model. (B) Focus on scenarios for which $s(M_v(0)) > 0$ and $s(M_u(\bar{v}))>0$. Asymptotic total abundance of $v$ obtained in simulations of the IDE model, as a function of $s(M_v(\bar{u}))$. The threshold for persistence of $\hat{\vb}$ in the harlequin model is indicated by the dotted line. Colors indicate the value of the theoretical IDE persistence criterion for $v$, namely $r(\mathcal T_v \circ \mathcal S_{\bar u})$.}
    \label{fig:harlequin}
\end{figure}

We numerically computed the prediction for both the IDE model and its harlequin approximation in the different scenario presented in Section \ref{subsec:numIDE}. We compared for every scenario the prediction of the harlequin approximation with the prediction of the IDE model and give the agreement table in Figure \ref{fig:harlequin}A. We observe that is most cases, the prediction of the harlequin approximation is similar to the prediction of the IDE model. However, when the harlequin approximation is wrong, it tends to predict more extinction than the true IDE model. For example, among all the scenario for which the harlequin's approximation predicted "SSE u" ($s(M_u(0) )> 0$ and $s(M_v(0)) \le 0$ ), 26$\%$ of them actually corresponded to the non-mutual invasion case "NMI u". We highlight that this miss-classification does not change the long time behaviour of the solutions, since in both cases "SSE u" and "NMI u" lead to the persistence of $u$ and extinction of $v$ (see Section \ref{sec:nmi}).
% \begin{figure}[ht]
%     \centering
%     \includegraphics[width = 0.4\textwidth]{figures/heatmap_prediction_v2.png}
%     \caption{Heatmap IDE vs HQ \manon{Ici il faudrait expliquer que la somme sur une ligne fait 1, ce qui veut dire }}
%     \label{fig:heatmap}
% \end{figure}

We then focus on the cases where $s(M_u(0))>0$ and draw on Figure \ref{fig:harlequin}B the limiting values of the total abundance of the solution $v$ of the IDE model as a function of $s(M_v(\overline \ub)) $. For each scenario, we color the obtained dot depending on $r(\mathcal T_v\circ S_{\bar u})$. The red colored dot situated on the left of the dotted line are scenarios for which the harlequin approximation predict the extinction of the $v$ population whereas this population survives in the IDE model. We note that many predictions errors of the harlequin approximation corresponds to cases where $r(\mathcal T_v\circ S_{\bar u})$ is close to $1$, which might explain why the approximation predicts $s(M_v(\overline \ub)) \le 0 $. 

%\madeleine{Finally, the single case of mutual uninvasibility studied in Section \ref{sec:mui} is not classified as MUI by the harlequin approximation, which instead predicts extinction of both species. However, as this scenario is near-critical for the IDE model, this mis-classification is unsurprising.} 
%\manon{Je pense que j'en n'aurais pas parlé ici en fait.}
\smallskip

As a conclusion, the harlequin approximation gives surprisingly good predictions for the limiting outcome of the IDE system, notably in the coexistence case. Furthermore all errors committed by the harlequin approximation go in the same direction: the approximation predicts the extinction of one or several species that survives in the IDE model. These two phenomena might be worth to explore from a theoretical perspective in future work.

\paragraph{Mutual uninvasibility} Finally, as we have not observed any MUI scenarios with the IDE model, we investigate MUI with the harlequin model. We consider a setting similar to the one above, with randomly sampled model parameters:
\begin{itemize}
    \item We fix $p_A = 0.5$.
    \item For each species, extinction rates per habitat type are sampled independently from a uniform distribution on $(0.003,0.01)$. Lower and upper bounds are given by the minimum and maximum approximate extinction rates $(\widehat{\tau}_X,\; X \in \{A,N\})$ computed in the previous section.
    \item For each species, colonization rates between each pair of habitat types are sampled independently from a uniform distribution on $(0.001,0.025)$. Lower and upper bounds are given by the minimum and maximum approximate extinction rates $(\widehat{c}_{XY},\; X, Y \in \{A,N\})$ computed in the previous section.
\end{itemize}

We sample independently 20000 parameter sets out of which 81 correspond to MUI scenarios (0.4$\%$), further emphasizing the rarity of this setting. Remarkably, all MUI scenarios are near-critical, as 
\[-0.001 < s(M_u(0)), s(M_v(0)), s(M_u(\bar \vb)), s(M_v(\bar \ub)) < 0.015. \]

For each MUI scenario, we consider 50 random initial conditions obtained as follows. We first sample from a uniform distribution on $(0,1)$ the proportion $q_X$ of occupied space in habitat $X$. Second, we sample the fraction $\rho_X$ of occupied habitat which is inhabited by species $u$ from a uniform distribution on $(0,1)$. This finally leads to $u_X(0) = q_X \rho_X$ and $v_X(0) = q_X(1-\rho_X)$. This procedure is executed independently for each habitat type. For each initial condition, we simulate the harlequin metacommunity until an equilibrium is reached (condition \eqref{eq:cvgc-crit}), which has occurred in all simulations. 

For each parameter set, we observe that starting from the majority of initial conditions, the dynamics converge to either monospecific equilibrium, illustrating the bistability described in Theorem \ref{thm:persistence_ODE}. However, in 58 out of all 81 MUI scenarios, starting from some initial conditions the dynamics do not converge to either monospecific equilibrium ($L^\infty$-distance to both $\bar \ub$ and $\bar \vb$ greater than $0.03$). Instead, both species persist and the system reaches a coexistence equilibrium. In addition, for 4 scenarios, there appear to be multiple coexistence equilibria (distance between at least two simulated coexistence equilibria greater than $0.03$ in $L^\infty$-distance). 
This opens an interesting research question, to provide a full characterization of the long time behaviour of the solutions in the harlequin case. This is left for future work, as we believe that it requires development of new approaches going beyond comparison arguments.

\section{Proofs for the integrodifferential model}
\label{sec:proofs_ide}
\subsection{Additional properties of the monospecific dynamics}
\label{sec:model1D}
% In this section, we consider the following equation:
% \begin{equation}
% %\label{eq:monospecific-ide}
% % \begin{aligned}
% \begin{cases}
%     \partial_t w(t,x) &= -\tau(x) w(t,x) + (1 - p(x) - w(t,x)) \int_\Omega c(x,y)w(t,y)dy, \\
%     \quad w(0,x) &= w_0(x).
% \end{cases}
% % \end{aligned}
% \end{equation}

In this section, we focus on the integro-differential equation \eqref{eq:monospecific-ide}.
We know from \cite[Proposition 2.9]{DDZ22} that the IDE \eqref{eq:monospecific-ide} is forward-invariant in
\[\Delta_p =\{ g \in L^{\infty}(\Omega, [0,1]) : g \leq 1 - p\}, \] 
and that the associated semi-flow is well defined.
%\manon{Les numéros d'équations ne sont pas cohérents}
More precisely, the semi-flow corresponds to the unique function 
\[ \phi_p : \R_+ \times \Delta_p \times \mathcal{C}(\Omega, (0,+\infty)) \times  \mathcal{C}(\Omega^2, (0,+\infty)]) \to \Delta_p \]
such that $t \mapsto \phi_p(t, w_0, \tau, c)$ is the unique solution to Equation \eqref{eq:monospecific-ide}. 

To start with we aim at verifying that the equation is well posed for initial conditions in 
\[\Delta_0 =\{ g \in L^{\infty}(\Omega, [0,1]) : g \leq 1 \}.\] 
\begin{lemma}
\label{lemma:con_ini}Assume that $0<p<1$. Consider any initial condition $w_0 \in \Delta_0$ and denote by $A\subset \Omega$ such that for all $x\in A$, $w_0(x)>1-p(x)$.\\
Then, equation \eqref{eq:monospecific-ide} admits a unique solution $w$ well defined on $[0,\infty[$ . Furthermore there exists a finite time $T_{A,w_0}$ after which the solution satisfies $w(t,\cdot)\in\Delta_p$ for all  $t\ge T_{A,w_0}$.
%\manon{Lemme à écrire, en gros si on prend une condition initiale dans $\Delta_0$, alors la solution reste bien définie pour tout temps. De plus, elle est positive et entre dans $\Delta_p$ en temps fini.}
\end{lemma}
\begin{proof}
%\madeleine{Dans nos hypothèses, on peut avoir $p(x) = 1$, en quel cas il y a de petits bugs aux étapes 1 (pas de colonisation possible en $x$) et 3 ($p(x) = 1$ donne $S_D(w) \leq \infty$); prendre $p \in (0,1)$?}

\textbf{Instantaneous propagation of the population.}
Let us consider an initial condition $w_0\in\Delta_0$ such that $\int_\Omega w_0(x) dx>0$. Since $c$ is continuous and positive on $\Omega$ compact, we have that for all $x\in\Omega$ 
\[\int_\Omega c(x,y)w_0(y)dy\ge  \min_{x,y\in\Omega ^2} c(x,y) \int_\Omega w_0(y)dy> 0.\]
%\textcolor{red}{Manon : ici ce n'est pas vraiment une question de continuité, c'est plutôt que l'on demande que $C(x,y)\ge c_{min}>0$}

As a consequence, if $x\in\Omega$ is such that $w_0(x)=0$, then 
\[ \partial_t w(0,x) = (1-p(x))\int_\Omega c(x,y)w_0(y)dy >0\]
and thus instantaneously the population is positive everywhere.

\smallskip

\noindent\textbf{Positivity of the solutions.}
We now want to verify that for any $t>0$ and $x\in\Omega$, $w(t,x)>0$. From the first step, we can assume that $w_0>0$.
The difficulty lies in the fact that we no longer assume that $w_0\le 1-p$.
For any $x\in\Omega$ and $s\ge0$, we have from \eqref{eq:monospecific-ide}
\begin{align*}
    (\partial_t w(s,x)+\tau(x)w(s,x))e^{\tau(x) s} = e^{\tau(x) s}(1-p(x)-w(s,x))\int_\Omega c(x,y) w(s,y)dy.
\end{align*}
For $0\le  t$, we deduce by integration that 
\begin{equation}
\label{eq:expr_mono}
w(t,x)  =w(0,x) e^{-\tau(x)t} + \int_0^t e^{-\tau(x)(t-s)}(1-p(x)-w(s,x))\int_\Omega c(x,y) w(s,y)dy ds.
\end{equation}

Let us assume by contradiction that
\[t_0=\inf\{ t\ge0, \exists x\in\Omega,  w(t,x)=0\},\] is finite, and denote by $x_0\in\Omega$ the associated position.
From \eqref{eq:expr_mono}, we obtain that 
\[0=w(0,x_0) e^{-\tau(x)t_0} + \int_0^{t_0} e^{-\tau(x_0)(t_0-s)}(1-p(x_0)-w(s,x_0))\int_\Omega c(x_0,y) w(s,y)dy ds.\]
This leads to a contradiction if for all $0\le s\le t_0$, $w(s,x_0)\le 1-p(x_0)$ as both terms of the right hand side are positive.
\\Otherwise, since $t\mapsto w(t,x_0)$ is continuous, we can define $s_0<t_0$ being the last time before $t_0$ such that $w(s,x_0)\ge 1-p(x_0)$. In that case, considering \eqref{eq:expr_mono} on $[s_0,t_0]$ gives the contradiction.\smallskip

\noindent\textbf{Finite time absorption in $\Delta_p$.}
Let us introduce the sets 
\[D_p(t)=\{x\in\Omega, w(t,x)>1-p(x)\}.\]

Let us first prove that if $x\notin D_p(0)$ then for all $t\ge0$, $x\notin D_p(t)$. By contradiction, let us assume that $t_p=\inf\{t\ge0, x\in D_p(t)\}$ is finite, then there exists an interval of the form $(t_p-\varepsilon,t_p]$ on which $t\mapsto w(t,x)$ increases. But by definition
$$\partial_t w(t_p,x) = -\tau(x) w(x,t_p) = -\tau(x)(1-p(x))<0, $$ 
which gives the contradiction.\\

Let us now consider $x\in D_p(0)$ and denote by 
\[S_D(x)=\inf\{ t\ge0, w(t,x)\notin D_p(t)\},\]
with the convention that $S_D(x)=\infty$ if $w(x,t)\le 1-p(x), \forall t\ge0$.\\ By definition for any $t\le S_D(x)$, we have $1-p(x)-w(t,x)\le 0$ and since $\int_\Omega c(x,y)w(t,y)dy \ge0$, we deduce that 
$$\partial_t w(t,x)\le -\tau(x) w(t,x),$$ which leads to 
$$w(t,x)\le w_0\exp(-\tau(x)t).$$
As a consequence \[S_D(x) \le \frac{1}{\tau(x)}\ln\left(\frac{w_0(x)}{1-p(x)}\right)<\infty.\]
\end{proof}

\paragraph{Comparison results.} 
We are interested in establishing monotonicity-type results of the semi-flow and persistence criterion in $p$, as this will be useful to control the metacommunity trajectories by monospecific ones. 
\begin{lemma}
\label{lem:r-comparison}
Let $p, q \in L^\infty(\Omega, [0,1])$ such that $p \leq q$.% and the set $\{x \in \Omega : p(x) = 1\}$ is Lebesgue-negligible. 
\begin{enumerate}[label=\alph*)]
    \item The spectral radii of the modified first generation operator introduced in  \eqref{eq:def_T} and \eqref{eq:def_S} satisfy \begin{equation} 
    r(\mathcal{T}_{c/\tau} \circ \mathcal{S}_p) \geq r(\mathcal{T}_{c/\tau} \circ \mathcal{S}_q), 
\end{equation}
\item For any $t > 0$ and any initial condition in $w_0 \in \Delta_{q}$, 
\begin{equation}
    \phi_p(t, w_0, \tau, c) \geq \phi_q(t, w_0, \tau, c).
\end{equation}
\end{enumerate}
\end{lemma}

\begin{proof} 
\noindent\textbf{Proof of $a)$.} By definition, the operators $A = \mathcal{T}_{c/\tau} \circ \mathcal{S}_p$ and $B = \mathcal{T}_{c/\tau} \circ \mathcal{S}_q$ are positive with respect to the cone $\mathfrak{K} = \{g \in L^{\infty}(\Omega, \R) : g \geq 0\}$ as $g \in \mathfrak{K} \implies (Ag \in \mathfrak{K} \text{ and } Bg \in \mathfrak{K})$. Since $p \leq q$, the operator $A - B$ also is positive with respect to $\mathfrak{K}$. Hence a) follows from Theorem 4.2 in \cite{marekFrobeniusTheoryPositive1970}. 
    
\noindent\textbf{Proof of $b)$.} For ease of notation, we let $\phi_p(t, w_0) = \phi_p(t, w_0, \tau, c)$ and define $\phi_q(t, w_0)$ analogously. By definition, we have $\phi_p(0, w_0) = \phi_q(0, w_0) = w_0$. As $w_0 \in \Delta_q \subset \Delta_p$, thus from the positive invariance, we deduce that for all $t\ge 0$, $\phi_p(t,w_0) \in\Delta_p$ and $\phi_q(t,w_0)\in\Delta_q$. Since $p \leq q$, it follows from Equation \eqref{eq:monospecific-ide} that 
\[\forall t \geq 0, \forall x \in \Omega, \quad \partial_t \phi_p(t, w_0)(x) \geq \partial_t \phi_q(t, w_0)(x).  \]
This concludes the proof.
\end{proof}

Finally, we require the following continuity result. 

\begin{lemma}
\label{lem:r-cont}
Let $p \in L^{\infty}(\Omega, [0,1])$ such that $\|p\|_\infty < 1$, and fix $\bar{\alpha} > 0$ satisfying $\bar{\alpha} p < 1$. Then the application $\alpha \mapsto r(\mathcal{T}_{c/\tau} \circ \mathcal{S}_{\alpha p})$ is continuous on $[0, \bar{\alpha}]$.
\end{lemma}

\begin{proof}
    Let $\alpha \in [0, \bar{\alpha}]$ and consider a sequence $(\alpha_n)_{n \geq 0}$ taking values in $[0, \bar{\alpha}]$ such that $\lim_n \alpha_n = \alpha $. For ease of notation, we let $\mathcal{A}_\alpha = \mathcal{T}_{c/\tau} \circ \mathcal{S}_{\alpha p}$ and $\mathcal{A}_n = \mathcal{T}_{c/\tau} \circ \mathcal{S}_{\alpha_n p}$. According to \cite[Lemma 2.1]{delmasEffectiveReproductionNumber2024}, if the family of linear operators $(\mathcal{A}_n, n \geq 0)$ is collectively compact and converges strongly to $\mathcal{A}$, then  
    \[ \lim_{n \to \infty} r(\mathcal{A}_n) = r(\mathcal{A}_\alpha).\] 
    It thus suffices to establish the desired properties of $(\mathcal{A}_n, n \geq 0)$.
    \smallskip
    
    \noindent \textbf{Collective compactness.} Recall that $(\mathcal{A}_n, n \geq 0)$ is said to be compact if the family
    \[ K = \{ \mathcal{A}_n g : n \geq 0, g \in L^\infty(\Omega, \R) : \|g\|_\infty \leq 1  \} \] 
    is relatively compact in $L^\infty(\Omega, \R)$.

    Let $\epsilon > 0$. By uniform continuity of $c$ on $\Omega^2$,
    \[\exists \eta > 0 : \|(x_1,y_1) - (x_2,y_2) \|_\infty < \eta \implies |c(x_1,y_1) - c(x_2,y_2)| < \frac{\epsilon}{\|\tau^{-1}\|_\infty}. \]
    Consider a family of disjoint sets $\{E_k, k = 1, \dots, N \}$ of diameter at most $\eta$ which satisfies $\cup_{k=1}^N E_k = \Omega$. Let $f \in K$ and $k \in \{1, \dots, N\}$. By definition, there exists $n \geq 0$ and $g \in L^{\infty}(\Omega, \R)$ with $\|g\|_\infty \leq 1$ such that $f = \mathcal{A}_n g$. Since $\alpha_n p \leq 1$, it follows that for any $x,z \in E_k$:
    \[|f(x) - f(z)| \leq \int_\Omega \frac{|c(x,y) - c(z,y)|}{\tau(y)}(1 - \alpha_n p(y))|g(y)| dy \leq \epsilon. \]
    Hence 
    \[ \forall k \in \{1, \dots, N\}, \forall f \in K, \quad \sup_{x,z \in E_k} |f(x) - f(z)| \leq \epsilon. \]
    Thus, $K$ is 
    %quasi-uniformly-equi-Lebesgue-measurable, which implies that it is 
    relatively compact in $L^\infty(\Omega, \R)$ \cite{cherkasCompactnessSpaces1970}. 
    \smallskip 
    
    \noindent \textbf{Strong convergence.}
    For any $g \in L^{\infty}(\Omega, \R)$, for any $n \geq 0$ and $x \in \Omega$,
    \[|\mathcal{A}_n g(x) - \mathcal{A}_\alpha g(x)| = \Big| \int_\Omega (\alpha_n - \alpha) \frac{c(x,y)}{\tau(y)} g(y) dy \Big| \leq ||c||_\infty \|g \tau^{-1}\|_\infty |\alpha_n - \alpha|. \]
    Hence, $\lim_{n \to \infty} \|\mathcal{A}_n g - \mathcal{A}_\alpha g \|_\infty = 0$, and thus the family $(\mathcal{A}_n, n \geq 0)$ converges strongly to $\mathcal{A}_\alpha$.
    This concludes the proof.
    
\end{proof}

\subsection{Proof of extinction}
\label{sec:proof_ext}
The proof of extinction relies on a comparison of the two populations $u$ and $v$ with monospecific dynamics. In order to simplify notations, we will denote by 
\[\phi_{u,p}(t,u_0) := \phi_p(t,u_0,\tau,c),\]
the semi-flow associated with \eqref{eq:monospecific-ide}, and by
\[ \phi_{v,p}(t,v_0):=\phi_p(t,v_0,\sigma,\gamma),\]
the semi-flow associated with \eqref{eq:monospecific-ide} when the extinction rate $\tau$ is replaced by $\sigma$ and the dispersion kernel $c$ by $\gamma$. Furthermore, we also use the notations $\phi_u = \phi_{u,0}$ and $\phi_v=\phi_{v,0}$.

We can now state a key comparison lemma.
\begin{lemma}[]
\label{lem:comparison}
Consider two functions $M$ and $\delta$ in $L^\infty(\Omega,[0,1])$ %\manon{? est-ce le bon espace ?} \madeleine{Je crois que oui, je ne vois pas de raison de demander mieux pour la preuve, et c'est assez large pour couvrir les cas qui nous intéressent}
, such that $M(x)>\delta(x)$, $\forall x\in\Omega$. 
Let us assume furthermore that for all $x\in\Omega$ and $t\ge0$,  $u(t,x)\in[\delta(x), M(x)]$. Then we can bound the second population $v$ as 
\[\phi_{v, \delta}(t,v_0)(x) \ge v(t,x) \ge \phi_{v,M}(t,v_0)(x),\quad \forall x\in\Omega, \forall t\ge0. \]
A similar result holds when $u$  is replaced by $v$ and conversely.
\end{lemma}

\begin{proof}
The proof relies on simple inequalities. We first focus on obtaining the upper bound. 
Let us recall that for all $x\in\Omega$
\begin{align*}
    \partial_t v(t,x) & = -\tau v(t,x) + (1-u(t,x)-v(t,x))\int v(t,y)c(x,y) dy\\
    &\le  -\tau v(t,x) + (1-\delta -v(t,x))\int v(t,y)c(x,y) dy.
\end{align*}
By integrating the inequality, we then obtain that 
\[ v(t,x) \le \phi_{v,\delta} (t,v_0)(x).\]
The lower bound derives from a similar argument.

\end{proof}

This lemma allows to establish Theorem \ref{theo:extin}.

\begin{proof}
[Proof of Theorem \ref{theo:extin}]
From Lemma \ref{lem:comparison} we obtain that since $v \ge0$, then 
\begin{equation}
    \label{eq:maj_u}
    0 \leq u(t,.) \le \phi_u(t,u_0).
\end{equation}
Similarly, since $u \ge 0$ as well, we have
\begin{equation}
    \label{eq:maj_v}
    0 \leq v(t,.) \le \phi_v(t,v_0).
\end{equation}
The result thus follows from Theorem \ref{theo:mono} since both $\phi_u$ and $\phi_v$ converge uniformly to $0$.

\end{proof}

\subsection{Proof of persistence}
\label{sec:proof_coexist}

\begin{proof}[Proof of Theorem \ref{theo:persist}] The proof is divided in several steps.\\

\noindent\textbf{Step 1} Let us first highlight that combining Lemma \ref{lem:r-comparison} with our assumption, we obtain that 
$$r(\mathcal T_u) \ge r(\mathcal T_u \circ S_{\bar v}) \ge 1.$$
As a consequence, Theorem \ref{theo:mono} (iii) shows that there exists $\bar{u} \in \mathcal{C}(\Omega, (0,1))$ such that 
\begin{equation}
\label{eq:uniform-cvgce}
    \lim_{t \to \infty} \| \phi_{u}(t, u_0) - \bar{u} \|_\infty = 0.
\end{equation} 

\noindent\textbf{Step 2} Since $v \geq 0$, we know from Lemma \ref{lem:comparison} that for any $t \geq 0$ and $x \in \Omega$, $u(t,x) \leq \phi_u(t, u_0)$. It thus follows from the uniform convergence in Equation \eqref{eq:uniform-cvgce} that 
\begin{equation} 
\label{eq:Talpha}
    \forall \alpha > 0, \exists T_\alpha >0 : \forall x \in \Omega, \quad u(t,x) \leq (1 + \alpha) \bar{u}(x).
\end{equation}
In addition, recall from Theorem \ref{theo:mono} (iii) that $\bar{u}$ is continuous and satisfies $\bar{u} < 1$. Thus, there exists $\alpha_0$ such that 
\[ \forall \alpha \leq \alpha_0, \forall x \in \Omega, \quad (1 + \alpha) \bar{u}(x) < 1. \]

\noindent\textbf{Step 3} Since $r(\mathcal{T}_v \circ \mathcal{S}_{\bar{u}}) > 0$, it follows from Lemma \ref{lem:r-cont} that we may pick $\alpha \leq \alpha_0$ small enough to ensure that 
\[ r(\mathcal{T}_v \circ \mathcal{S}_{(1 + \alpha)\bar{u}}) > 0. \]
Considering the associated $T_\alpha$ as defined in \eqref{eq:Talpha}, it follows again from Lemma \ref{lem:comparison} that 
\begin{equation}
\label{eq:minorer-v}
    \forall t \geq T_\alpha, \forall x \in \Omega, \quad v(t,x) \geq  \phi_{v, (1 + \alpha) \bar{u}} (t - T_\alpha, v(T_\alpha, \cdot)) (x).
\end{equation}  

Notice that $v(T_\alpha, x) > 0$ for any $x \in \Omega$. Indeed, Equation \eqref{eq:ide-system_dim2} ensures that for any $t \geq 0$ and $x \in \Omega$,
\[ \partial_t v(t,x) \geq -\| \tau\|_\infty v(t,x), \]
and thus $v(t,x) \geq \mathrm{e}^{-\|\tau\|_\infty t} v_0(x) > 0$. 

Hence, Theorem \ref{theo:mono} implies the existence of $\tilde{v} > 0$ such that 
\[ \lim_{t \to \infty} \phi_{v, (1 + \alpha)\bar{u}}(t - T_\alpha, v(T_\alpha, \cdot)) = \tilde{v} > 0.  \]
As $\bar{u}$ is continuous, $\tilde{v}$ also is continuous according to Theorem \ref{theo:mono} (iii). Thus Equation \eqref{eq:minorer-v} yields the existence of $\epsilon > 0$, independent from $v(T_\alpha, \cdot)$ and thus $(u_0, v_0)$, such that 
\[ \liminf_{t \to \infty} v \geq \epsilon > 0. \]
As the roles of $u$ and $v$ are interchangeable, this concludes.
\end{proof} 
% \manon{Il faut noter ici que $\epsilon$ est indépendant de la condition initiale. Reformuler correctement le théorème et mettre le terme de uniform persistence et la référence au bouquin de Smith and Thieme.}

\subsection{Proof of single species extinction}

\begin{proof}[Proof of Theorem \ref{thm:single-extinct}]
Since $v \geq 0$, Lemma \ref{lem:comparison} ensures that 
\[0 \leq u(t, \cdot) \leq \phi_u(t, u_0). \]
Hence, since we assume that $r(\mathcal T_u)\le 1$ 
\begin{equation}
\label{eq:aux-cvunif}
    \|u(t,\cdot)\|_\infty \leq \|\phi_u(t,u_0)\|_\infty \xrightarrow[t \to \infty]{} 0.
\end{equation} 
Further, our assumption that $r(\mathcal T_v) > 1$ combined with Lemma \ref{lem:r-cont} ensures that there exists $\alpha_0 > 0$ such that, $r(\mathcal{T}_v \circ \mathcal{S}_{\alpha}) > 1$ for any $\alpha \leq \alpha_0$. Let $\alpha \in (0, \alpha_0)$. It follows from Equation \eqref{eq:aux-cvunif} that there exists $T_\alpha \geq 0$ such that for any $t \geq T_\alpha$, 
\[ \| u(t,\cdot)  \|_\infty \leq \alpha. \]
Proceeding as in Step 3 of the Proof of Theorem \ref{theo:persist}, we obtain that for any $t \geq T_\alpha$,
\[ v(t, \cdot) \geq \phi_{v , \alpha}(t - T_\alpha, v(T_\alpha, \cdot)) \xrightarrow[t \to \infty]{} \bar{w}_\alpha, \]
with $\bar{w}_\alpha$ the unique monospecific equilibrium of $v$ for $p(x) = \alpha$.
As a consequence, it follows that 
\[ \forall \alpha \in (0, \alpha_0), \forall x \in \Omega, \quad \liminf_{t \to \infty} v(t,x) \geq \bar{w}_\alpha(x). \]

Notice that Lemma \ref{lem:comparison} implies that for any $x \in \Omega$, the application $\alpha \mapsto \bar{w}_\alpha(x)$ is decreasing on $(0, \alpha_0)$. As $\bar{w}_\alpha$ is bounded, this ensures that $\bar{w}_0(x) = \lim_{\alpha \to 0+} \bar{w}_\alpha(x)$ is well defined for any $x \in \Omega$. Further, recall that $\bar{w}_\alpha$ satisfies 
\[\forall x \in \Omega, \quad 0 = -\sigma(x) \bar{w}_\alpha(x) + (1 - \alpha - \bar{w}_\alpha(x)) \int_\Omega \gamma(x,y) \bar{w}_\alpha(y) dy.  \]
Thus, letting $\alpha$ go to zero leads to 
\[\forall x \in \Omega, \quad 0 = -\sigma(x) \bar{w}_0(x) + (1 - \bar{w}_0(x)) \int_\Omega \gamma(x,y) \bar{w}_0(y) dy.  \]
As a consequence, Theorem \ref{theo:mono} $(iii)$ implies that $\bar{w}_0 = \bar{v}$. Thus 
\[\forall x \in \Omega, \quad \liminf_{t \to \infty} v(t,x) \geq \bar{v}(x). \]

Finally, since $u \geq 0$, Lemma \ref{lem:comparison} implies that $v(t, \cdot) \leq \phi_v(t, u_0)$, whence
\[ \forall x \in \Omega, \quad \limsup_{t \to \infty} v(t,x) \leq \bar{v}(x). \]
This concludes the proof.
\end{proof}

\subsection{Proof of bistability}

\begin{proof}
[Proof of Proposition \ref{prop:bistable}]
Let us first assume that the initial conditions satisfy 
$|\bar u - u_0| \le \delta \bar u$ and $|v_0|\le \eta$ for $\delta,\eta>0$ to be set latter. We define a stopping time 
$$\tau=inf\{t\ge0,   u(t,\cdot) < (1 - \delta)\bar u $$% \text{ and } v(t,\cdot) > \eta\}$$
Our goal is to prove that for a good choice of $\delta,\eta$, then $\tau=+\infty$.

Before $\tau$, we can compare the solutions $u$ and $v$ with monomorphic populations. For the population $v$, we have 
$$v\le \phi_{v, (1 - \delta)\bar u}.$$
Moreover, from Lemma \ref{lem:r-cont},  $r(\mathcal T_v \circ S_{(1 - \delta)\bar u}) <1$ if $\delta $ is small enough. In that case, $\phi_{v, (1 - \delta)\bar u}$ is decreasing and converges to $0$. In particular, before time $\tau$, $v\le \phi_{v, (1 - \delta)\bar u}< v_0 \le \eta$.

For population $u$, we have 
$$\phi_{u,\eta} \le u ,$$
and choosing $\eta$ small enough leads to $r(\mathcal T_u \circ S_{\eta})>1$. As a consequence $\phi_{u,\eta}$ increases toward $\bar u$ as $t\to\infty$.

From this two bounds, we deduce that necessarly $\tau=+\infty$ and furthermore $v\to0$. 
To prove the convergence of $u$ let us recall that since $v\ge0$, we always have $u\le \phi_u$ which converges to $\bar u$ and leads to the conclusion using the squeeze theorem.

\end{proof}
%Objectif  : Montrer que $\tau=+\infty$ si les conditions initiales, $\delta$ et $\eta$ sont bien choisis.

%Rem :  on a toujours $u \le \phi_u$ qui converge vers $\bar u$.

%Extinction of $v$$$v\le \phi_{v, (1 - \delta)\bar u}$$  and we want $r(\mathcal T_v \circ S_{(1 - \delta)\bar u}) <1$ which is true if $\delta $ small by Lemma \ref{lem:r-cont}.Then $\phi_{v, (1 - \delta)\bar u}$ is decreasing and converges to $0$. Note that we prove that $v\le v_0$.

%Convergence of $u$$$\phi_{u,\eta} \le u $$Si $\eta$ petit $r(\mathcal T_u \circ S_{\eta}>1$, et donc $\phi_{u,\eta}$ converge de façon monotone vers $\bar u$, donc elle ne peut pas franchir le seuil.

%Donc c'est bon.

\section{Proof for the harlequin model}
\label{sec:proofs_hq}

In this section, we handle the case of a discrete space $\Omega=\Omega_A\sqcup \Omega_N$ and consider solutions $(\ub,\vb)=(u_A,u_N,v_A,v_N)$ to \eqref{eq:ode-system}. %\manon{Du point de vue des notations, je me demanque si on a pas interet à souligner l'aspect bi-dimentionnel en mettant en gras ou avec une flèche comme un vecteur $\mathbf{u} = (u_A,u_N)$ ou bien $\overrightarrow{u}=(u_A,u_N)$. Tu en penses quoi ?\\ }

The proof strategy is very similar to the continuous framework, using comparison between the two species system and the monospecific one. We will use the same notations for the solutions of the monospecific systems, but adapted to the reduce space. Namely we denote by $\Phi_u(t, \ub^0)= (\phi_{u,A}(t, \ub^0),\phi_{u,N}(t, \ub^0))$ the solution of 
\begin{equation}
\label{eq:monoEDO}
    \left\{\begin{aligned}
    w_A' &= -\tau_A w_A + (p_A - w_A) (c_{AA}w_A + c_{AN} w_N), \\
    w_N' &= -\tau_N w_N + (p_N - w_N) (c_{NA}w_A + c_{NN} w_N). \\
    & (w_A(0),w_N(0)) =  \ub^0.
    \end{aligned}\right.
\end{equation}
and $\Phi_v(t, \ub^0)$ the solution when parameters $(\tau_A,\tau_N)$ are changed to $(\sigma_A,\sigma_N)$ and $(c_{AA}, c_{AN},c_{NA}, c_{NN})$ to $(\gamma_{AA}, \gamma_{AN},\gamma_{NA}, \gamma_{NN})$.  

%\manon{Si on veut transposer notre méthode de preuve au cas d'espace discret il faut \begin{itemize}
%    \item vérifier que les solutions ont un sens si on démarre au dessus de $p_X$
%    \item pouvoir comparer les rayons spectraux des matrices
%    \item avec un résultat de continuité de l'équilibre également en fonction des paramètres $p_X$ au moins.
%end{itemize}}

\begin{proof}[Proof of Theorem \ref{thm:persistence_ODE}]
\noindent\textbf{Proof of $(i)$ --}
As for the continuous case, we notice that since $\ub,\vb\ge0$, we have $\forall t\ge0$ 
\begin{equation}
    \label{eq:upperbound_mono}
    \ub(t)\le \Phi_u(t, \ub^0)\qquad \text{ and } \qquad \vb(t)\le \Phi_v(t,\vb^0).
\end{equation} We deduce the result from Theorem \ref{thm:Laj} since both $ \Phi_u(t, \ub^0)$ and $ \Phi_v(t,\vb^0)$ converge to $(0,0)$ as $t \to\infty$.\\
\noindent\textbf{Proof of $(ii)$ --} 
Assume that $s(M_u(0))\le0$, then for any $ \ub^0>0$, $\Phi_u(\ub^0,t)\to(0,0)$ as $t\to\infty$. Using the same upper-bound we easily obtain that $\ub(t) \to (0,0)$ as well. For the second population, since $s(M_v(0))>0$, using the continuity of the spectral bound with respect to the entries of the matrix $M$, there exists $\varepsilon>0$ small enough such that $s(M_v(\varepsilon))>0$. Furthermore, from the convergence of $u$ to $0$, there exists a time $T_\varepsilon$ such that for all $t\ge T_\varepsilon$, $\ub(t)\le \varepsilon$. Therefore for all $t\ge T_\varepsilon$, and $X\in\{A,N\}$ 
\begin{align*}
    \frac{d}{dt} v_X & = -\sigma_X v_X + (p_A-u_X-v_X) (\gamma_{AX} v_A+\gamma_{NX}v_N)\\
   & \ge -\sigma_X v_X + (p_A-\varepsilon-v_X) (\gamma_{AX} v_A+\gamma_{NX}v_N)
\end{align*}
By integration, we obtain that for all $t\ge0$, $$\vb(t+T_\varepsilon)\ge \Phi_v^\varepsilon(t+T_\varepsilon, v(T_\varepsilon))$$ where $\Phi_v^\varepsilon$ is the flow of the monomorphic system where the proportion of agricultural and natural space are changed to $p_X+\varepsilon$, for $X\in\{A,N\}$.
Since $s(M_v(\varepsilon))$ and $\vb(T_\varepsilon)>0$, $\Phi_v^\varepsilon$ converges to a positive equilibrium $\bar \vb^\varepsilon$. As a consequence 
\[\liminf_{t\to\infty } \vb(t) \ge \bar \vb^\varepsilon,\]
and from \eqref{eq:upperbound_mono}
\[ \bar \vb \ge \limsup_{t\to\infty } \vb(t).\]
The conclusion of the proof is obtained by letting $\varepsilon\to0$.\\
\noindent\textbf{Proof of $(iii)$} Once again, the main idea is similar to the continuous space case. 
From \eqref{eq:upperbound_mono}, and our assumption, we deduce that for any $\alpha>0$ there exists a time $T_\alpha$ such that  for all $t\ge T_\alpha$, and $X\in\{A,N\}$
\[ u_X(t)\le \bar u_X (1+\alpha),\quad \text{and} \quad v_X(t) \le \bar v_X (1+\alpha) .\]
As a consequence we will lower bound $\ub$ and $\vb$ by solutions of the monomorphic system with modified values of $p_A,p_N$. 
More precisely, let us consider $\Phi_u^\alpha$ the solution of the dynamical system 
\begin{equation*} 
    \left\{\begin{aligned}
    (w_A^\alpha)' &= -\tau_A w_A^\alpha + (p_A -\bar v_A(1+\alpha) - w_A^\alpha) (c_{AA}w_A^\alpha + c_{AN} w_N^\alpha), \\
    (w_N^\alpha)' &= -\tau_N w_N^\alpha + (p_N -\bar v_N(1+\alpha) - w_N^\alpha) (c_{NA}w_A^\alpha + c_{NN} w_N^\alpha). \\
    & (w_A^\alpha(0),w_N^\alpha(0)) = u(T_\alpha).
    \end{aligned}\right.
\end{equation*}
Then for all $t\ge T_\alpha$
$$\ub(t-T_\alpha) \ge \Phi_u^\alpha(t).$$
A similar construction provides a solution of the monomorphic system $\Phi_v^\alpha$ such that
\[\vb(t-T_\alpha)\ge \Phi_v^\alpha(t),\quad \forall t\ge T_\alpha.\]
Now, we can choose $\alpha$ small enough such that both $s(M_u(\bar \vb (1+\alpha)))>0$ and $s(M_v(\bar \ub (1+\alpha)))>0$, which ensures that $$\Phi_u^\alpha(t) \longrightarrow_{t\to\infty} \bar{\ub}^\alpha>0,$$ and $$\Phi_v^\alpha(t) \longrightarrow_{t\to\infty} \bar{\vb}^\alpha>0.$$
From this we obtain the persistence of the populations, \textit{i.e.} there exists $\varepsilon>0$ such that for any initial condition $\ub_0 + \vb_0 \le 1$, 
$$\liminf_{t\to\infty} \ub(t) \ge\varepsilon,\quad \text{and}\quad \liminf_{t\to\infty} \vb(t) \ge\varepsilon.$$
\noindent\textbf{Proof of $(iv)$} The proof follows exactly the same reasoning as in the continuous setting and its adaptation is left to the reader.
\end{proof}

It remains to establish the existence of a coexistence equilibrium. 

\begin{proof}[Proof of Theorem \ref{thm:exist_eq_ODE}]
The proof relies on the construction of a compact convex invariant by the dynamical system. Indeed, \cite[Lemma 4.1]{lajmanovich_deterministic_1976} guarantees that this is a sufficient condition for the existence of an equilibrium within this invariant set.

The construction of the invariant set relies on comparison with the monospecific system and additional properties of its vector-field. Let us consider the solution $\Phi_u$ of \eqref{eq:monoEDO} and denote by $F$ the associated vector field such that
$$\nabla \Phi_u(t, u^0) = F(\Phi_u(t, u^0)).$$ 
We obtain, using the formalism of \cite{ducrot_differential_2022}, that the vector field is positive and monotonous (\cite{ducrot_differential_2022} Theorem 7.1 for positivity and Theorem 8.23 for monotonicity). These properties are similar to the cooperativeness obtained for the integro-differential system in \cite{DDZ22}. In particular, we deduce from Corrolary 8.5 in \cite{ducrot_differential_2022}, that the solutions are non decreasing as soon as the initial condition $u^0$ satisfies $ F(u^0)>0$ and non increasing if $ F(u^0)<0$. Combining this with Theorem \ref{thm:Laj}, we deduce that the rectangle $[0, \bar u_A]\times [0, \bar u_N]$ is positively invariant by the monomorphic dynamics. \\
Using the coupling $\eqref{eq:upperbound_mono}$, we deduce that for any initial condition $(u^0, v^0) \in [0, \bar u_A]\times [0, \bar u_N]\times [0, \bar v_A]\times [0, \bar v_N]$, the solutions of the dimorphic system remains in this set. 

We now aim at constructing a positive lower bound for each coordinate. Similarly as for the persistence, we deduce that for any $t\ge0$, $\ub(t)\ge \Phi_u^0(t, u^0)$ where $\Phi_u^0$ is solution to 
\begin{equation*}
    \left\{\begin{aligned}
    (w^0_A)' &= -\tau_A w^0_A + (p_A -\bar v_A - w^0_A) (c_{AA}w^0_A + c_{AN}w^0_N), \\
    (w^0_N)' &= -\tau_N w^0_N+ (p_N -\bar v_N - w^0_N) (c_{NA}w^0_A + c_{NN} w^0_N). \\
    &\Phi_u^0(0, u^0)=u^0
    \end{aligned}\right.
\end{equation*}
From our assumptions, since $s(M_u(\bar v))>0$ the solution $\Phi_u^0(t, u^0)$ converges as $t\to\infty$ towards a positive equilibrium $\bar{u}^0$. Furthermore, from the monotonicity of the solutions, if the initial condition is below $\bar{u}^0$, then the solution $\Phi_u^0(t, u^0)$ is non decreasing. 
As a consequence, with a symmetrical argument for the $\vb$ populations, we deduce that there exists $ \eta>0$ such that $[\eta, \bar u_A]\times [\eta, \bar u_N] \times [\eta, \bar v_A]\times [\eta, \bar v_N]$ is positively invariant for the two dimensional system, which concludes the proof.
 
\end{proof}
\appendix
\section{Proof of the graphon approximation}
\label{app:preuve_graphon}
This section is devoted to the proof of Theorem \ref{thm:lln}.
The proof proceeds in several steps. We start by showing that the sequence of distributions of $(\eta^K)_{K \geq 1}$ is C-tight in $\mathbb{D}(\R_+, \mathcal{M}_1(E))$, which means that it is tight and its adherence values are almost surely continuous. Second, we establish that all adherence values almost surely satisfy a particular deterministic measure-valued equation. By showing that the solution $\zeta$ to the latter is unique, we conclude that $(\eta^K)_{K \geq 1}$ converges in probability to $\zeta$. Finally, we prove that $\zeta$ indeed is defined as in Equation \eqref{eq:def-zeta}, which ends the argument. 
\smallskip

Let us start with some preliminary computations that will be useful throughout the section. We will work with the semimartingale decomposition of $\eta^K$, for $K > 0$ fixed. For $i \in \{0,1\}$, let 
\[\widetilde{Q}_i =  {Q}_i - \mu_i\]
be the compensated martingale-measure associated to $Q_i$. With this notation, it follows that for any $t \geq 0$ and $f \in \mathcal{B}_b(E, \R)$, 
\[ \angles{\eta^K_t}{f} = M^K_t(f) + V^K_t(f), \]
where the martingale and bounded variation parts are respectively defined by 
\begin{equation*}
%\label{eq:semimgle-m}
\begin{aligned}
    M^K_t&(f) = \frac{1}{K} \int_0^t \int_{E_0} \setind{\theta \leq \tau_{w_k(s-)}(x_k)} \big(f(x_k, 0) - f(x_k, w_k(s-))   \big)\widetilde{Q}_0(ds, dk, d\theta) \\ 
    & + \frac{1}{K} \int_0^t \int_{E_1} \setind{\theta \leq K^{-1} c_{w_k(s-), w_\ell(s-)}(x_k, x_\ell)} \big(f(x_\ell, w_k(s-)) - f(x_\ell, w_\ell(s-))  \big)\widetilde{Q}_1(ds, dk,d\ell, d\theta), 
\end{aligned}
\end{equation*}
and 
\begin{equation}
\label{eq:semimgle-v}
\begin{aligned}
    V^K_t&(f) = \angles{\eta^K_0}{f} + \frac{1}{K} \sum_{k=1}^K \int_0^t \tau_{w_k(s)} \big(f(x_k,0) - f(x_k, w_k(s)) \big) ds \\
    & + \frac{1}{K^2} \sum_{k, \ell = 1}^K \int_0^t c_{w_k(s), w_\ell(s)}(x, y) \big( f(x_\ell, w_k(s)) - f(x_\ell, w_\ell(s)) \big) ds.
\end{aligned}
\end{equation}
We start with a brief technical lemma.

\begin{lemma}
\label{lemma:crochet}
    Let $f \in \mathcal{B}_b(E, \R)$, $t \geq 0$ and $K \geq 1$. Under Assumption \ref{hyp:jump-rates}, $M^K_t(f)$ is a square-integrable martingale, whose quadratic variation is given by 
    \begin{equation}
    \label{eq:crochet}
    \begin{aligned}
        \langle M^K(f) \rangle_t &= \frac{1}{K^2}\sum_{k=1}^K \int_0^t \tau_{w_k(s)} \big(f(x_k,0) - f(x_k, w_k(s)) \big)^2 ds \\
    & + \frac{1}{K^3} \sum_{k, \ell = 1}^K \int_0^t c_{w_k(s), w_\ell(s)}(x, y) \big( f(x_\ell, w_k(s)) - f(x_\ell, w_\ell(s))^2 \big) ds.
    \end{aligned}
    \end{equation}
\end{lemma}

\begin{proof}
 Let $f \in \mathcal{B}_b(E, \R)$, $t \geq 0$ and $K \geq 1$. We have 
\begin{equation*}
\begin{aligned}
    \E&[ \langle M^K(f) \rangle_t] = \E\left[\int_0^t \int_{E_0}  \left(\frac{1}{K} \setind{\theta \leq \tau_{w_k(s)}(x_k)} \big(f(x_k, 0) - f(x_k, w_k(ss))   \big) \right)^2 \mu_0(ds, dk, d\theta) \right] \\
    & + \E \left[\int_0^t \int_{E_1} \left( \frac{1}{K} \setind{\theta \leq K^{-1} c_{w_k(s), w_\ell(s)}(x_k, x_\ell)} \big(f(x_\ell, w_k(s)) - f(x_\ell, w_\ell(s))  \big) \right)^2  \mu_1(ds, dk, d\ell, d\theta) \right] \\
    &= \E \left[ \frac{1}{K^2}\sum_{k=1}^K \int_0^t \tau_{w_k(s)} \big(f(x_k,0) - f(x_k, w_k(s)) \big)^2 ds \right] \\
    & + \E \left[ \frac{1}{K^3} \sum_{k, \ell = 1}^K \int_0^t c_{w_k(s), w_\ell(s)}(x, y) \big( f(x_\ell, w_k(s)) - f(x_\ell, w_\ell(s))\big)^2 ds \right] \\
    & \leq \frac{C}{K} \|f \|_\infty t,
\end{aligned}   
\end{equation*}
where 
\begin{equation}
\label{eq:aux-defC}
    C = 4 \Big(\max_{w \in \bbrackets{1}{S}} \|\tau_w \|_\infty + \max_{w,w' \in \bbrackets{1}{S}} \|c_{w,w'} \|_\infty \Big) < +\infty,
\end{equation}
which is a finite constant thanks to Assumption \ref{hyp:jump-rates} and compacity of $\Omega$.

Thus $M^K_t(f)$ is square-integrable, with quadratic variation given by
\begin{equation*}
\begin{aligned}
    \langle & M^K(f) \rangle_t = \int_0^t \int_{E_0}  \left(\frac{1}{K} \setind{\theta \leq \tau_{w_k(s)}(x_k)} \big(f(x_k, 0) - f(x_k, w_k(ss))   \big) \right)^2 \mu_0(ds, dk, d\theta) \\
    & + \int_0^t \int_{E_1} \left( \frac{1}{K} \setind{\theta \leq K^{-1} c_{w_k(s), w_\ell(s)}(x_k, x_\ell)} \big(f(x_\ell, w_k(s)) - f(x_\ell, w_\ell(s))  \big) \right)^2  \mu_1(ds, dk, d\ell, d\theta),
\end{aligned}
\end{equation*}
leading to Equation \eqref{eq:crochet}.
\end{proof}
\smallskip

We are now ready to establish the desired tightness result.

\begin{prop}
    Under Assumption \ref{hyp:jump-rates}, the sequence of distributions of $(\eta^K)_{K \geq 1}$ is C-tight in $\mathbb{D}(\R_+, \mathcal{M}_1(E))$.
\end{prop}

\begin{proof}
Following \cite{delmas_individual-based_2024}, we aim at establishing tightness using the following criterion from \cite[Theorem II.4.1]{bolthausenLecturesProbabilityTheory2002}. Recall that a set of functions $\mathcal{D} \subset \C_b(E, \R_+)$ is separating if for any measures $\mu, \nu \in \mathcal{M}_1(E)$,
\[(\forall f, \angles{\mu}{f} = \angles{\nu}{f}) \implies \mu = \nu. \]
Then in order to establish that the sequence $(\eta^K)_{K \geq 1}$ is C-tight in $\mathbb{D}(\R_+, \mathcal{M}_1(E))$, it suffices to show that the following two conditions hold:
    \begin{enumerate}
        \item \emph{Compact containment}. For any $T \geq 0$ and $\varepsilon > 0$, there exists a compact subset $K_{T, \varepsilon} \subseteq E$ such that 
        \[\sup_{K \geq 1} \P \left( \sup_{t \leq T} \eta^K_t(K^C_{T, \varepsilon}) > \varepsilon \right) < \varepsilon. \]
        \item \emph{Tightness of projections.} There exists a separating set $\mathcal{D}$ containing the constant functions such that for any $f \in \mathcal{D}$, the sequence of processes $(\angles{\eta^K_\cdot}{f})_{K \geq 1}$ is C-tight in $\mathbb{D}(\R_+, \R)$.
    \end{enumerate}
Notice that compact containment is always satisfied using $K_{T, \varepsilon} = E$, since $E$ is compact itself. It thus only remains to establish C-tightness of projections. 
\smallskip

Let $f \in \mathcal{C}_b(E, \R)$. We start by showing tightness of $(\angles{\eta^K_\cdot}{f})_{K \geq 1}$ in $\mathbb{D}(\R_+, \R)$. According to the Aldous-Rebolledo criterion \cite{aldousStoppingTimesTightness1978,joffeWeakConvergenceSequences1986}, it is enough to show that:
\begin{enumerate}[label=(\alph*)]
    \item For any time $t$ belonging to a dense subset of $\R_+$, the sequences $(\langle M^K(f) \rangle_t)_{K \geq 1}$ and $(V^K_t(f))_{K \geq 1}$ are tight.
    \item For any $T \geq 0$, for any $\varepsilon, \alpha > 0$, there exists $\delta > 0$ and $K_0 \geq 1$ such that for any two sequences of stopping times $(S_K)_{K \geq 1}$ and $(T_K)_{K \geq 1}$ satisfying $S_K \leq T_K \leq T$ for all integers $K$, 
    \begin{equation*}
    \begin{aligned}
        \sup_{K \geq K_0} \P \left( |\langle M^K(f) \rangle_{T_K} - \langle M^K(f) \rangle_{S_K} | \geq \alpha, \; T_K \leq S_K + \delta \right) \leq \varepsilon, \\
        \text{and } \sup_{K \geq K_0} \P \left( | V^K_{T_K}(f) - V^K_{S_K}(f) | \geq \alpha, \; T_K \leq S_K + \delta \right) \leq \varepsilon.
    \end{aligned}
    \end{equation*}
\end{enumerate}
Let us check that both conditions are satisfied. First, consider any $t > 0$. Notice that it follows from the proof of Lemma \ref{lemma:crochet} and Equation \eqref{eq:aux-defC} that there exists $C > 0$ (independent from $K$, $t$ and $f$) such that
\[ \sup_{K \geq 1} \E[| \langle M^K(f) \rangle|_t] < \frac{C}{K}\|f\|_\infty t < \infty.\]
Similarly, letting 
\[ C' = 2 \Big(\max_{w \in \bbrackets{1}{S}} \|\tau_w \|_\infty + \max_{w,w' \in \bbrackets{1}{S}} \|c_{w,w'} \|_\infty \Big), \]
it follows from Equation \eqref{eq:semimgle-v} that 
\[ \sup_{K \geq 1} \E[|V^K_t(f)] < C'\|f\|_\infty t < \infty. \]
Thus, for any $t \in \R_+$, the sequences $(\langle M^K(f) \rangle_t)_{K \geq 1}$ and $(V^K_t(f))_{K \geq 1}$ are tight, and condition (a) is met.

Fix $T, \varepsilon, \alpha > 0$ and consider now two sequences of stopping times $(S_K)_{K \geq 1}$ and $(T_K)_{K \geq 1}$ as described in (b). Assume that almost surely, there exists some $\delta > 0$ such that for any $K \geq 1$, $T_K \leq S_K + +\delta$. Letting $C$ be defined by Equation \eqref{eq:aux-defC}, we obtain that 
\[
\E[ |\langle M^K(f) \rangle_{T_K} -  \langle M^K(f) \rangle_{S_K}| ] \leq \E[ |T_K - S_K| ] \frac{C}{K} \| f\|_\infty \leq \frac{C}{K} \| f\|_\infty \delta.  
\]
Similarly, using the constant $C'$ defined above, it holds that
\[
\E[|V^K_{T_K}(f) - V^K_{S_K}(f)|] \leq C' \|f\|_\infty \delta.
\]
Using conditional Markov's inequality, it follows that for any $K \geq 1$,
\begin{equation*}
\begin{aligned}
     \P \left( |\langle M^K(f) \rangle_{T_K} - \langle M^K(f) \rangle_{S_K} | \geq \alpha, \; T_K \leq S_K + \delta \right) \leq \frac{C \|f\|_\infty^2 \delta}{K \alpha}, \\
    \P \left( | V^K_{T_K}(f) - V^K_{S_K}(f) | \geq \alpha, \; T_K \leq S_K + \delta \right) \leq \frac{C' \|f\|_\infty \delta}{\alpha}.
\end{aligned}    
\end{equation*}
Thus, choosing first $\delta$ small enough such that $C' \|f\|_\infty \delta/\alpha < \varepsilon$, and second $K_0$ large enough to ensure that $C \|f\|_\infty^2 \delta/(K_0 \alpha) < \varepsilon$ suffices to proof (b). Thus $(\angles{\eta^K_\cdot}{f})_{K \geq 1}$ is tight in $\mathbb{D}(\R_+, \R)$.

Finally, notice that by definition, there exists $c >0$ such that for any $T \geq 0$, the following inequality holds almost surely:  
\[ \sup_{t \leq T} |\angles{\eta^K_t}{f} - \angles{\eta^K_{t-}}{f}| \leq \frac{c}{K}.\]
According to \cite[Proposition VI.3.26]{jacodLimitTheoremsStochastic2003}, the sequence $(\angles{\eta^K_\cdot}{f})_{K \geq 1}$ thus is actually C-tight in $\mathbb{D}(\R_+, \R)$, which concludes the proof.
\end{proof}

Tightness of $(\eta^K)_{K \geq 1}$ in $\mathbb{D}(\R_+, \mathcal{M}_1(E))$ ensures that this sequence admits some adherence values. The next step consists in showing that the latter satisfy some measure-valued equation.

\begin{prop}
\label{prop:identification}
    Under Assumptions \ref{hyp:jump-rates} and \ref{hyp:eta0}, any limiting value $\zeta = (\zeta_t)_{t \geq 0}$ of $(\eta^K)_{K \geq 1}$ in $\mathbb{D}(\R_+, \mathcal{M}_1(E))$ is almost surely solution to the following system of measure-valued equations. For any measurable bounded function $f$, for any $t \geq 0$,
    \begin{equation}
    \label{eq:def-zeta-mv}
    \begin{aligned}
        \angles{\zeta_t}{f} &= \angles{\zeta_0}{f} + \sum_{w = 1}^S \int_0^t \int_\Omega \tau_w(x) \left( f(x,0)-f(x,w) \right)\zeta_s(dx, w) ds \\
        &+ \sum_{w, w' = 0}^S \int_0^t \int_\Omega \int_\Omega c_{w,w'}(x,y) \left( f(y,w)-f(y,w') \right) \zeta_s(dx, w) \zeta_s(dy, w') ds,
    \end{aligned}
    \end{equation}
    with initial condition $\zeta_0$ defined by Equation \eqref{eq:zeta0}.
\end{prop} 

\begin{proof}
Fix $t  \geq 0$ and $f \in \mathcal{B}_b(E, \R)$. Define the following function on $\mathbb{D}(\R_+, \mathcal{M}_1(E))$: for any $\eta \in \mathbb{D}(\R_+, \mathcal{M}_1(E))$,
\begin{equation*}
\begin{aligned}
    \psi_{f,t}(\eta) & = \angles{\eta_t}{f} - \angles{\eta_0}{f} - + \sum_{w = 1}^S \int_0^t \int_\Omega \tau_w(x) \left( f(x,0)-f(x,w) \right)\eta_s(dx, w) ds \\
        &+ \sum_{w, w' = 0}^S \int_0^t \int_\Omega \int_\Omega c_{w,w'}(x,y) \left( f(y,w)-f(y,w') \right) \eta_s(dx, w) \eta_s(dy, w') ds.
\end{aligned}
\end{equation*}
Recall that, by definition,
\[\eta^K_t(dx, dw) = \frac{1}{K} \sum_{k=1}^{K} \delta_{(x_k, w_k(t))}(dx, dw).\]
It thus follows from the semimartingale decomposition of $\eta^K$, and in particular Equation \eqref{eq:semimgle-v}, that 
\[ \psi_{f,t}(\eta^K) = M^K_t(f). \]
Hence, letting $C$ be the positive constant defined in Equation \eqref{eq:aux-defC}, 
\[\E[|\psi_{f,t}(\eta^K)|]^2 \leq \E[|\psi_{f,t}(\eta^K)|^2] = \E[\langle M^K(f) \rangle_t] \leq \frac{C}{K}\|f\|_\infty t \xrightarrow[K \to \infty]{} 0.\]

Further, one may notice that the sequence $(\psi_{f,t}(\eta^K))_{K \geq 1}$ is uniformly bounded because $\eta^K_t \in \mathcal{M}_1(E)$. Indeed, 
\[ |\psi_{f,t}(\eta^K)| \leq (1 + \max_{w \in \bbrackets{1}{S}} \|\tau_w \|_\infty + \max_{w,w' \in \bbrackets{1}{S}} \|c_{w,w'} \|_\infty) 2\| f \|_\infty t. \]
Hence the sequence $(\psi_{f,t}(\eta^K))_{K \geq 1}$ is uniformly integrable. 

Consider an adherence value $\zeta$ of $(\eta^K)_{K \geq 1}$. Then there exists some subsequence $(\eta^{\phi(K)})_{K \geq 1}$ which converges to $\zeta$ in $\mathbb{D}(\R_+, \mathcal{M}_1(E))$. If it holds that $(\psi_{f,t}(\eta^{\phi(K)}))_{K \geq 1}$ converges in law to $\psi_{f,t}(\zeta)$, then it follows from the uniform integrability of the former that 
\[\E[|\psi_{f,t}(\zeta)|] = \lim_{K \to \infty} \E[|\psi_{f,t}(\eta^{\phi(K)})|] = 0,  \]
which concludes the proof.

It thus only remains to show that $(\psi_{f,t}(\eta^{\phi(K)}))_{K \geq 1}$ converges in law to $\psi_{f,t}(\zeta)$. The continuity hypothesis of Assumption \ref{hyp:jump-rates} ensures that $\psi_{f,t}$ is continuous at any $\eta \in \mathbb{D}(\R_+, \mathcal{M}_1(E))$ such that for any $s \in [0,t]$, the marginal $\eta_s(dx \times \bbrackets{0}{S})$ is absolutely continuous with respect to the Lebesgue measure. 

By construction, for any $K \geq 1$ and $s \geq 0$, it holds that $\eta^K_s(dx \times \bbrackets{0}{S}) = \eta^K_0(dx \times \bbrackets{0}{S})$. Further, Assumption \ref{hyp:eta0} ensures that $(\eta^K_0(dx \times \bbrackets{0}{S}))_{K \geq 1}$ converges in distribution in $\mathcal{M}_1(\Omega)$ to $\mu(x) dx$. Continuity of the application $\nu \mapsto \nu(dx \times \bbrackets{0}{S})$ on $\mathcal{M}_1(E)$, and continuity of the trajectories of $\zeta$ (because of C-tightness) allow to conclude that for any $t \geq 0$, 
\[ \zeta_t(dx \times \bbrackets{0}{S}) = \mu(x) dx. \]
Thus $\psi_{f,t}$ is continuous at $\zeta$, which implies that $(\psi_{f,t}(\eta^{\phi(K)}))_{K \geq 1}$ converges in law to $\psi_{f,t}(\zeta)$. This ends the proof.
\end{proof}

The last step of the tightness-identification-uniqueness argument consists in establishing uniqueness of the adherence values of $(\eta^K)_{K \geq 1}$ in $\mathbb{D}(\R_+, \mathcal{M}_1(E))$, establishing its convergence in $\mathbb{D}(\R_+, \mathcal{M}_1(E))$. 

\begin{prop}
\label{prop:uniqueness}
    Let $\zeta, \overline{\zeta}$ be two solutions of Equation \eqref{eq:def-zeta-mv} starting from the same initial condition. It then holds that 
    \begin{equation*}
       \forall t \geq 0, \quad \|\zeta_t - \overline{\zeta}_t\|_{TV} = 0.
    \end{equation*}
\end{prop}

\begin{proof}
   Let $f \in \mathcal{B}_b(E, \R)$ such that $\|f\|_\infty \leq 1$, and consider two solutions $\zeta, \overline{\zeta}$ of Equation \eqref{eq:def-zeta-mv} starting from the same initial condition. It follows from Equation \eqref{eq:def-zeta-mv} that for any $t \geq 0$, 
   \begin{equation*}
       |\angles{\zeta_t - \overline{\zeta}_t}{f} | \leq \int_0^t ( A_1(s) + A_2(s) + A_3(s)) ds,
   \end{equation*}
   with 
   \begin{equation*}
   \begin{aligned}
        A_1(s) &= \Big| \sum_{w=1}^S \int_\Omega \tau_w(x)\big( f(x,0) - f(x,w) \big) (\zeta_s - \overline{\zeta}_s)(dx, w) \Big|, \\
        A_2(s) &= \Big| \sum_{w,w'=1}^S \int_\Omega \int_\Omega c_{w,w'}\big( f(y,w) - f(y,w') \big) \big(\zeta_s(dx, w) - \overline{\zeta}_s(dx, w) \big) \zeta_s(dy, w') \Big|, \\
        \text{and } A_3(s) &= \Big| \sum_{w,w'=1}^S \int_\Omega \int_\Omega c_{w,w'}\big( f(y,w) - f(y,w') \big) \overline{\zeta}_s(dx, w) \big(\overline{\zeta}_s(dy,w') - \zeta_s(dy, w') \big) \Big|.
   \end{aligned}
   \end{equation*}
   
   Recall that for any bounded measurable function $g$, by definition of the total variation norm, 
   \[ |\angles{\zeta_s - \overline{\zeta}_s}{g}| = \Big| \sum_{w=1}^S \int_\Omega g(x,w)\big(\zeta_s - \overline{\zeta}_s\big)(dx,w)\Big| \leq \|g\|_\infty \|\zeta_s - \overline{\zeta}_s\|_{TV}. \]
   
   Since for any $w$, $\tau_w$ is continuous on $\Omega$ (and thus in particular measurable and bounded), there thus exists $C_1 > 0$ such that, for any $s \in [0,t]$,
   \[ A_1(s) \leq C_1 \| \zeta_s - \overline{\zeta}_s \|_{TV}. \]
   Similarly, the application
   \begin{equation*}
       (x, w) \in E \mapsto \sum_{w' = 1}^S \int_\Omega c_{w,w'}\big( f(y,w) - f(y,w') \big) \zeta_s(dy, w')
   \end{equation*}
   is measurable and bounded. Thus, there exists $C_2 > 0$ such that for any $s \in [0,t]$, 
   \[ A_2(s) \leq C_2 \| \zeta_s - \overline{\zeta}_s \|_{TV}. \]
   Reasoning analogously, we also obtain existence of $C_3 > 0$ such that for any $s \in [0,t]$, 
   \[ A_3(s) \leq C_3 \| \zeta_s - \overline{\zeta}_s \|_{TV}. \]
   
   Hence, letting $C = C_1 + C_2 + C_3 > 0$, we finally obtain that  
   \[ |\angles{\zeta_t - \overline{\zeta}_t}{f} | \leq C \int_0^t \| \zeta_s - \overline{\zeta}_s \|_{TV} ds.\]
   It follows from Gronwall's lemma (\cite[Appendix Theorem 5.1]{ethierMarkovProcessesCharacterization1986}) that 
   \[ |\angles{\zeta_t - \overline{\zeta}_t}{f} |= 0. \]
   Since $t$ is arbitrary, this concludes the proof.
\end{proof}

Now that we are sure that $(\eta^K)_{K \geq 1}$ converges to the unique solution to Equation \eqref{eq:def-zeta-mv}, it remains to show that the deterministic process defined in Equation \eqref{eq:def-zeta} indeed is uniquely characterized by Equation \eqref{eq:full-ide-system}, and solves this measure-valued equation.

\begin{prop}
    The integro-differential system \eqref{eq:full-ide-system} admits a unique solution, which provides the unique solution $\zeta$ of Equation \eqref{eq:def-zeta-mv} through Equation \eqref{eq:def-zeta}.
\end{prop}

\begin{proof}
Recall from the proof of Proposition \ref{prop:identification} that for any $t \geq 0$, 
    \[ \zeta_t(dx \times \bbrackets{0}{S}) = \mu(x)dx, \]
which means that $\zeta_t$ is absolutely continuous with respect to the following measure on $E$: 
\[ dx \otimes \sum_{i=0}^S \delta_i(dw). \]   
Thus, there exist $u_0, \dots u_S \in \mathcal{B}(\R_+ \times E, [0,1])$  such that $\sum_{i = 0}^S u_i \leq 1$ and 
\[ \zeta_t(dx, dw) = \mu(x) dx \sum_{i=0}^S u_i(t,x) \delta_i(dw). \]

Let $i \in \bbrackets{1}{S}$ and $g \in \mathcal{B}_b(\Omega, \R)$. Applying Equation \eqref{eq:def-zeta-mv} to $f(x,w) = g(x) \setind{w=i}$ leads to
\begin{equation*}
    \int_\Omega F_i(t,x) g(x) \mu(x) dx = 0,
\end{equation*}
where 
\begin{equation*}
\begin{aligned}
    F_i(t,x) & = u_i(t,x) - u_0(t,x) - \tau_i(x) u_i(t,x) + \sum_{j = 1}^S u_j(t,x) \int_\Omega u_i(t,y) c_{ij}(y,x) \mu(y) dy \\
    & - \sum_{j = 1}^S u_i(t,x) \int_\Omega u_j(t,y) c_{ji}(y,x) \mu(y) dy.
\end{aligned}
\end{equation*}
Since $g$ is arbitrary and $\mu > 0$, it follows that
\[ F_i(t,x) = 0 \quad \forall x \in \Omega.\]
As this holds for any $i \in \bbrackets{1}{S}$ and $t \geq 0$, and $u_0 = 1 - \sum_{i=1}^S u_i$, we conclude that the family $(u_i, i \in \bbrackets{1}{S})$ indeed solves Equation \eqref{eq:full-ide-system}. 

In addition, any solution to Equation \eqref{eq:full-ide-system} provides a solution to Equation \eqref{eq:def-zeta-mv} through Equation \eqref{eq:def-zeta}. Thus, Proposition \ref{prop:uniqueness} implies that Equation \eqref{eq:full-ide-system} admits a unique solution. 
\end{proof}

\section*{Funding}
The authors are partially funded by the Chair ”Modélisation Mathématique et Biodiversité” of Veolia Environnement-École Polytechnique-Muséum national d’Histoire naturelle-Fondation X and by ANR project HAPPY (ANR-23-CE40-0007).

\section*{Acknowledgements}
The authors are grateful to Jean-Fran\c cois Delmas and Pierre-Andr\' e Zitt for stimulating discussions.

\end{document}